\documentclass[11pt]{amsart}

\usepackage{amsthm}
\usepackage{amsmath}
\usepackage{amssymb}
\usepackage{color}

\newcommand{\tr}{\mathop\mathrm{tr}}
\newcommand{\rk}{\mathop\mathrm{rank}}

\newcommand{\spn}{\mathop\mathrm{span}}
\newcommand{\R}{\mathop\mathrm{Re}}
\newcommand{\I}{\mathop\mathrm{Im}}

\newtheorem{thm}{Theorem}[section]
\newtheorem{theorem}[thm]{Theorem}

\newtheorem{corollary}[thm]{Corollary}
\newtheorem{problem}[thm]{Problem}

\newtheorem{notation}[thm]{Notation}
\newtheorem{example}[thm]{Example}

\newtheorem{lemma}[thm]{Lemma}
\newtheorem{proposition}[thm]{Proposition}

\newtheorem{definition}[thm]{Definition}
\theoremstyle{remark}
\newtheorem{remark}[thm]{Remark}

\newtheorem{ex}[thm]{Example}

\newcommand{\ZZ}{\mathbb Z}
\newcommand{\RR}{\mathbb R}
\newcommand{\MM}{\mathbb M}
\newcommand{\NN}{\mathbb N}
\newcommand{\CC}{\mathbb C}
\newcommand{\HH}{\mathbb H}
\newcommand{\BB}{\mathbb B}

\newcommand{\pc}[1]{{\color{magenta}{#1}}}

\begin{document}

\title{The solution to the frame Quantum Detection Problem}
\author[Botelho-Andrade, Casazza, Cheng, Tran
 ]{Sara Botelho-Andrade, Peter G. Casazza, Desai Cheng and Tin T. Tran}
\address{Department of Mathematics, University
of Missouri, Columbia, MO 65211-4100}

\thanks{The authors were supported by
 NSF DMS 1609760; NSF ATD 1321779; and ARO  W911NF-16-1-0008}

\email{sandrade102087@gmail.com, Casazzap@missouri.edu}
\email{
chengdesai@yahoo.com, tinmizzou@gmail.com}


\begin{abstract}
We will give a complete solution to the frame quantum detection problem.  We will solve both cases of the problem: the quantum injectivity problem
and quantum state estimation problem.  We will answer the problem in both
the real and complex cases and in both the finite dimensional and
infinite dimensional cases.

\noindent Finite Dimensional Case:

\begin{enumerate}
\item We give two complete classifications of the 
sets of vectors which solve the injectivity problem - for both the real and
complex cases.  We also give methods for constructing them.
\item We show that the frames which solve the injectivity problem
are open and dense in the family of all frames.
\item We show that the Parseval frames which give injectivity are dense in the Parseval frames.
\item We classify all frames for which the state estimation problem is solvable, and when it is not solvable, we give the best approximation to a
solution.
\end{enumerate}
\noindent Infinite Dimensional Case:
\begin{enumerate}
\item We give a classification of all frames which solve the injectivity problem
and give methods for constructing solutions.
\item We show that the frames solving the injectivity problem are neither open
nor dense in all frames.
\item We give necessary and sufficient conditions for a frame
to solve the state
estimation problem 
for all measurements in $\ell_1$ and show that there is no injective frame for which the state estimation problem is solvable
for all measurements in $\ell_2$.
\item When the state estimation problem does not have an exact solution, we give the best approximation to a
solution.
\end{enumerate}
\end{abstract}

\maketitle
\pagebreak
\tableofcontents

\pagebreak

\section{Introduction and Preliminaries}

In this paper we will give a complete answer to the frame quantum detection problem including the injectivity problem
and state estimation problem.  We will answer the problem in both
the real and complex cases and in both the finite dimensional and
infinite dimensional cases.
\vskip10pt
\noindent {\bf Important Notation}.  Throughout the paper we will
let $\{e_i\}_{i=1}^n$ be the canonical orthonormal basis of
$\RR^n$ or $\CC^n$ and $\{e_i\}_{i=1}^{\infty}$ will denote the
canonical orthonormal basis of real or complex $\ell_2$.  Also, $\iota$ will be used to denote the complex unit.
	
	For a vector $x_k$ in $\RR^n$ or $\CC^n$, we denote its coordinates as
	$$x_k=(x_{k1}, x_{k2}, \ldots, x_{kn}).$$
	Similarly, for $x_k$ belonging to $\ell_2$, we write
	$$x_k=(x_{k1}, x_{k2}, \ldots, x_{ki}, \ldots).$$
To explain exactly what we will solve, we need to introduce the basics of quantum detection.
Let $L^{\infty}(\HH)$ be the space of bounded linear operators on
a finite or infinite dimenional (real or complex) Hilbert space
$\HH$.  Let $\{e_i\}_{i\in I}$ be an orthonormal basis for $\HH$.
For an operator  $T\in L_0(\HH)$, the finite rank operators on $\HH$, the {\it trace} of $T$ is given by: $\tr(T)= \sum_{i\in I}\langle
Te_i,e_i\rangle$, which is finite and independent of the orthonormal
basis.  The 
trace induces a scalar product by $\langle T,S\rangle_{HS}=
\tr(TS^*)$.  The closure of $L_0(\HH)$ with respect to this scalar product, denoted $L^2(\HH)$ is the space of the Hilbert-Schmidt
operators on $\HH$.  For any $T\in L^{\infty}(\HH)$ we denote by
$|T|=\sqrt{TT^*}$, the positive square root of $TT^*$.  We say that 
$T$ is a {\it trace class operator} if $\tr(|T|)<\infty$.  The
set of all trace class operators is denoted by $L^1(\HH)$ and forms a Banach space under the 
{\it trace norm} $\|T\|_1=\tr(|T|)$. 

Let 
\[ Sym(\HH) = \{T:T\in L^{\infty}(\HH),\ T=T^*\},\]
denote the real Banach space of self-adjoint operators on $\HH$ and
let 
\[ Sym^+(\HH)= \{T =T^*\ge 0\},\]
denote the real cone of positive self-adjoint operators on $\HH$.  
The main objects to analyze these operators are the {\it positive
operator-valued measures}.

\subsection{Positive Operator-Valued Measures}

In quantum mechanics, the definition of a von Neumann measurement
can be generalized using positive operator-valued measures (POVMs)
\cite{Eld,Eld2,H}.  The advantage of this is that it allows one to distinguish
more accurately among elements of a set of non-orthogonal quantum
states.

Let $X$ denote a set of outcomes (e.g. this could be a finite or
infinite subset of $\ZZ^d$ or $\RR^d$).  Let $\beta$ denote a 
sigma algebra of subsets of $X$.

\begin{definition}
A {\bf positive operator-valued measure} ({\bf POVM}) is a
function $\Pi:\beta \rightarrow Sym^+(\HH)$ satisfying:
\begin{enumerate}
\item $\Pi(\emptyset)=0$ (the zero operator).
\item For every disjoint family $\{U_i\}_{i\in I}\subset \beta$,
$x,y \in \HH$ we have 
\[ \left \langle \Pi\left ( \cup_{i\in I}U_i\right )x,y \right 
\rangle= \sum_{i\in I}\langle \Pi(U_i)x,y\rangle.\]
\item $\Pi(X)= I$ (the identity operator).
\end{enumerate}
\end{definition}

\subsection{Quantum Systems}

A {\bf quantum system} is defined as a von Neumann algebra $\mathcal{A}$
of operators acting on $\HH$.  The set of {\bf states} on $\HH$
is 
\[ \mathcal{S}(\HH)=\{T\in L^1(\HH),\ T=T^*\ge 0,\ \tr(T)=1\},\]
and it represents the reservoir of {\bf quantum states} for any
quantum system.  

The set of {\bf quantum states} $\mathcal{S}(\mathcal{A})$ associated to a 
quantum system $\mathcal{A}$ is obtained by identifying states that differ by a null state with respect to $\mathcal{A}$.  Thus, the
set of quantum states are in one-to-one correspondance with the 
linear functionals on $\mathcal{A}$ of the form:
\[ \rho:\mathcal{A}\rightarrow \CC,\mbox{ for some }
S\in \mathcal{S}(\HH),\ \rho(T)=\tr(TS),\ \mbox{ for every }
T\in \mathcal{A}.\]

A quantum state $\rho \in \mathcal{S}(\mathcal{A})$ is called
a {\bf pure state} if it is an extreme point in the convex
$weak^*$ compact set of quantum states $\mathcal{S}(\mathcal{A})$.
We say a POVM $\Pi$ is {\it associated to a von Neumann algebra} $\mathcal{A}$
if $\Pi:\beta \rightarrow \mathcal{A}\cap Sym^+(\HH)$.  

Given a quantum state $\rho$, the {\bf quantum measurement} performed by the POVM $\Pi$ is the map $p: \beta \to \RR$ defined by $p(U)=\rho(\Pi(U))=\tr(\Pi(U)T)$, where $T\in \mathcal{S}(\HH)$ is in the equivalence class associated to $\rho$.

\subsection{The Quantum Detection Problem}

Let $L(\beta,\RR)$ denote the set of bounded functions
defined on $\beta$.  Given a POVM $\Pi$ associated to a von Neumann algebra $\mathcal{A}$, the {\bf quantum detection problem} is formulated as follows.
\vskip12pt
\noindent {\bf Quantum Detection Problem}.  Is there a unique quantum state $\rho \in  \mathcal{S}(\mathcal{A})$ compatible with the set
of quantum measurements performed by the POVM $\Pi$?

Specifically, the quantum detection problem asks two questions:
\begin{enumerate}
\item {\it Injectivity, or state separability}:  Is the following map injective
\[ \mathbb{M}: \mathcal{S}( \mathcal{A})\rightarrow L(\beta,\RR),\ \ \ 
\mathbb{M}(\rho)(U)=\rho(\Pi(U))?\]
\item {\it Range analysis, or state estimation}:  Assume $\mathbb{M}$ is injective.  Then, given a map $p \in L(\beta,\RR)$, determine if $p$ is in the range of $\mathbb{M}$, hence is of the form $p=\mathbb{M}(\rho)$ for some unique $\rho \in  \mathcal{S}( \mathcal{A})$.  If not, find a quantum state $\rho$ that best approximates $p$ in some sense (e.g. robustness to noise).
\end{enumerate}

We point out that in the context of quantum detection in
quantum mechanics, a significant amount of work has been put into
computing the probability of detection error 
\cite{Eld3,H,HW,P,Y}.  We will not
address this question here.

\subsection{Frame POVMs}

In this section we introduce the 
Hilbert space frame version of the quantum detection problem.
For a background on frame POVMs we recommend \cite{BK,Eld,HL,Coc}.
For a background on Hilbert space frame theory we recommend
\cite{CK,C1,Ch}.

\begin{definition}
A family of vectors $\{x_k\}_{k\in I}$ is a {\bf frame} for a 
real or complex, finite
or infinite dimensional Hilbert space $\HH$ if there
are constants $0<A\le B<\infty$ satisfying:
\[ A\|x\|^2 \le \sum_{k\in I}|\langle x, x_k\rangle|^2 \le 
B\|x\|^2,\mbox{ for all }x\in \HH.\]
\end{definition}

We have
\begin{enumerate}
\item $A,B$ are the {\bf lower and upper frame bounds} of the frame.
\item If $A=B$ this is a {\bf tight frame}.  If $A=B=1$ this is
a {\bf Parseval frame}.
\item If we only assume we have $0<B<\infty$, this is called
a {\bf B-Bessel sequence}. Note that $\|x_k\|^2 \le B$, for all
$k\in I$.
\end{enumerate}

We define the {\bf analysis operator} of the frame as
$T:\HH \rightarrow \ell_2(I)$ by 
\[ T(x)=(\langle x,x_1\rangle,\langle x,x_2\rangle,\ldots)
=\sum_{k\in I}\langle x,x_k\rangle e_k.\]
The {\bf synthesis operator} $T^*$ is given by:  
\[ T^*\left ( \{a_k\}_{k\in I}\right ) = \sum_{k\in I}a_kx_k.\]
The {\bf frame operator} is $S=T^*T$.  This is a positive,
self-adjoint invertible operator on $\HH$ satisfying:
\[ S(x)= \sum_{k\in I}\langle x,x_k\rangle x_k.\]
It is known that for any frame $\{x_k\}_{k\in I}$, $\{S^{-1/2}x_k\}_{k\in I}$ is a Parseval frame.  It is also known that a frame
is Parseval if and only if its frame operator is the identity
operator.

\begin{definition}
	A frame $\{x_k\}_{k\in I}$ is said to be {\bf bounded} if there is a constant $ C>0$ such that 
	\[ \|x_k\|\ge C, \mbox{ for all } k\in I.\]
\end{definition}

If $\{x_k\}_{k\in I}$ is a Parseval frame for a Hilbert space 
$\HH$, it naturally induces a POVM $\Pi$ on $X=I$ with
$\beta=2^{I}$ (the power set of $I$):
\[ \Pi(U) =\sum_{k\in U}x_kx_k^*,\mbox{ where }
x_k^*:\HH \rightarrow \CC,\ x_k^*(x)=\langle x,x_k\rangle,\]
with strong convergence for any $U\subset I$.

Given a state $T\in \mathcal{S}(\HH)$ (i.e. a unit-trace, trace class, positive, self-adjoint operator on $\HH$), the frame induced quantum measurement
is given by the function
\[ p:\beta\rightarrow \RR,\ \ p(U)=\sum_{k\in U}\tr(Tx_kx_k^*)
= \sum_{k\in U}\langle Tx_k,x_k\rangle.\]
 For the von Neumann algebra $\mathcal{A}=L^{\infty}(\HH)$, the quantum states coincide with the convex set of states $\mathcal{S}(\HH)$.  
 In this case, the injectivity problem and the state estimation problem ask:
\vskip12pt
\noindent {\bf Injectivity Problem}:
 Is there
 a Parseval frame $\chi=\{x_k\}_{k\in I}$ so that the map $\MM:\mathcal{S}(\HH)\rightarrow L(\beta,\RR)$ defined by $\MM(T)(U)=\sum_{k\in U}\langle Tx_k,x_k\rangle$ for $U\subset I$ is injective? 

\vskip10pt
 \noindent{\bf State Estimation Problem}: Given an injective Parseval frame  $\{x_k\}_{k\in I}$ and a function $p: \beta\to \RR$, is there any $T\in \mathcal{S}(\HH)$ so that $\MM(T)=p$? If not, find a quantum state $T$ that best approximates $p$.

\subsection{Generalizing Quantum Detection}

We will work on a much more general quantum detection problem.
In particular, we will work with 
\begin{enumerate}
\item Self-adjoint operators which
may not be positive.
\item Operators which are not trace one but are Hilbert Schmidt.
\item Frames which are not Parseval.
\end{enumerate}

We will see that solving the problem in this more general form
will also solve the original problem.

 First, we need a definition.
\begin{definition}
	A family of vectors $\mathcal{X}=\{x_k\}_{k\in I}$ in a Hilbert space $\HH$ is said to be {\bf injective} if whenever a Hilbert Schmidt self-adjoint operator $T$ satisfies 
	\[ \langle Tx_k, x_k\rangle=0,
	\mbox{  for all }k\in I,\] then $T=0$.
\end{definition}

Now we will show that we do not need to find Parseval frames for the quantum detection problem. If we have a frame giving injectivity, then its canonical Parseval frame is injective.

\begin{proposition}
	Let $\{x_k\}_{k\in I}$ be a frame for $\HH$ which gives injectivity. If $F$ is a bounded invertible operator 
	on $\HH$, then $\{Fx_k\}_{k\in I}$ also gives injectivity.
\end{proposition}

\begin{proof}
	Let $T$ be a Hilbert Schmidt self-adjoint operator such that $$\langle TFx_k, Fx_k\rangle=0, \mbox{ for all } k.$$ Then $\langle F^*TFx_k, x_k\rangle=0$, for all $k$. Note that $F^*TF$ is also a Hilbert Schmidt self-adjoint operator. Therefore, $F^*TF=0$ and hence $T=0$. 	
\end{proof}

\begin{corollary}\label{cor1}
	Let $\{x_{k}\}_{k\in I}$ be a frame with frame operator $S$. If $\{x_{k}\}_{k\in I}$ gives injectivity, then the canonical Parseval frame $\{S^{-1/2}x_k\}_{k\in I}$ also gives injectivity.
\end{corollary}



\section{The Solution for the Finite Dimensional Case}

In this section we will solve the finite dimensional injectivity problem and the state estimation problem for both
the real and complex cases.  These problems were originally
solved by Scott \cite{S} (See also \cite{BH}) where the solutions
are called {\bf informationally complete quantum measurements}.
We will have to redo this here since we need much more information
about the solutions and need proofs in a format that will easily
generalize to infinite dimensions.

\subsection{Solution to the Injectivity Problem}

 First, we will see that we do not need to work with positive operators via the following theorem.
\begin{theorem}\label{positivefinite}
	Given a family of vectors $\mathcal{X}=\{x_k\}_{k=1}^m$ in $\HH^n$, the following are equivalent:
	\begin{enumerate}
		\item Whenever $T, S$ are positive and self-adjoint, and
		\[ \langle Tx_k,x_k\rangle = \langle Sx_k,x_k\rangle,
		\mbox{ for all k},\]
		then $T=S$.
		\item Whenever $T, S$ are self-adjoint, and
		\[ \langle Tx_k,x_k\rangle = \langle Sx_k,x_k\rangle,
		\mbox{ for all k},\]
		then $T=S$.
		\item $\mathcal{X}$ is injective.
	\end{enumerate}
\end{theorem}
\begin{proof}
	$(1)\Rightarrow (2)$: Let $T, S$ be self-adjoint operators such that 	\[ \langle Tx_k,x_k\rangle = \langle Sx_k,x_k\rangle,
	\mbox{ for all } k.\]
	Set $$m_1:=\inf_{\Vert x\Vert=1}\langle T x, x\rangle, \quad m_2:=\inf_{\Vert x\Vert=1}\langle S x, x\rangle$$ then $m_1, m_2\in \RR.$ Set $m=\min\{m_1, m_2\}$.
	
	Now let $P=T-mI$, $Q=S-mI$. Then for any $x\in \HH, \Vert x\Vert=1$, we have
	\begin{align*}
	\langle Px, x\rangle=\langle (T-mI)x, x\rangle=\langle Tx, x\rangle - m\geq 0.
	\end{align*}
	Hence, $P$ is positive. Similarly, $Q$ is positive.
	
	We have 
	\begin{align*}
	\langle Px_k, x_k\rangle&=\langle (T-mI)x_k, x_k\rangle\\
	&=\langle Tx_k, x_k\rangle - m\Vert x_k\Vert^2\\
	&=\langle Sx_k, x_k\rangle - m\Vert x_k\Vert^2\\
	&=\langle Qx_k, x_k\rangle.
	\end{align*}
	By (1) we get $P=Q$ and therefore $T=S$.
	
	$(2)\Rightarrow (3)$:  Let $T$ be any self-adjoint operator such that
	\[ \langle Tx_k,x_k\rangle =0,\mbox{ for 
		all } k.\]
	Then 
	\[ \langle Tx_k,x_k\rangle= \langle Sx_k,x_k\rangle,\mbox{ for 
		all } k,\] where $S=0$. It follows that $T=0$.
	
	$(3)\Rightarrow (1)$: Let any positive self-adjoint operators $T, S$ satisfy
	\[ \langle Tx_k,x_k\rangle=\langle Sx_k,x_k\rangle,\mbox{ for 
		all } k.\]
	Then
	\[ \langle (T-S)x_k,x_k\rangle =0,\mbox{ for 
		all } k.\]
	Since $T-S$ is a self-adjoint operator, $T=S$ by (3).
\end{proof}

\begin{remark}
	If we further require that the operators are trace one, then to
	prove injectivity, we only need to show that if $T$ is trace
	zero and $\langle Tx_k,x_k\rangle =0$ for all $k=1, 2, \ldots$, then
	$T=0$.  Since if $T,S$ are trace one and
	\[ \langle Tx_k,x_k\rangle =\langle Sx_k,x_k\rangle,
	\mbox{ for all }k.\]
	then
	\[ \langle (T-S)x_k,x_k\rangle =0,\mbox{ for all }k
	\mbox{ and }\tr(T-S)=0.\]
\end{remark}

\subsubsection{The real case}
We start with a propositon which
shows where our classification of the quantum
detection problem comes from.

\begin{proposition}\label{prop1}
	Given a self-adjoint operator $T=(a_{ij})_{i,j=1}^n$ on $\RR^n$ and
	a vector $x=(x_1,x_2,\ldots,x_n)\in \RR^n$, we have
	\[ \langle Tx,x\rangle  = \sum_{i=1}^n \sum_{j=1}^n a_{ij}x_ix_j=
	\sum_{i=1}^na_{ii}x_i^2 + 2\sum_{i=1}^n\sum_{j=i+1}^na_{ij}x_ix_j.
	\]
\end{proposition}

\begin{proof}
	First we compute:
	\[ Tx = \left (\sum_{j=1}^na_{1j}x_j,\sum_{j=1}^na_{2j}x_j,\ldots,
	\sum_{j=1}^n a_{nj}x_j\right ).\]
	So,
	\[ \langle Tx,x\rangle =  \sum_{j=1}^na_{1j}x_1x_{j} +
	\sum_{j=1}^na_{2j}x_2x_j+\cdots + \sum_{j=1}^n a_{nj}x_nx_j.
	\]
	Using the fact that $T$ is self-adjoint:
	\begin{align*}
	\langle Tx,x\rangle &= \sum_{i=1}^n \sum_{j=1}^n a_{ij}x_ix_j\\
	&= \sum_{i=1}^na_{ii}x_i^2 + \sum_{1\leq i<j\leq n}a_{ij}x_ix_j
	+ \sum_{1\leq j<i\leq n}a_{ij}x_ix_j\\
	&= \sum_{i=1}^na_{ii}x_i^2 + \sum_{1\leq i<j\leq n}a_{ij}x_ix_j
	+ \sum_{1\leq i<j\leq n}a_{ji}x_jx_i\\
	&= \sum_{i=1}^na_{ii}x_i^2 + 2\sum_{1\leq i<j\leq n}a_{ij}x_ix_j.
	\end{align*}
\end{proof}

This proposition leads us to the following definition:

\begin{definition}\label{def real finite}
To a vector $x=(x_1, x_2, \ldots, x_n)\in \RR^n$ we associate
	a vector $\tilde{x}$ in $\RR^{\frac{n(n+1)}{2}}$ by:
	$$\tilde{x}=(x_1x_1, x_1x_2, \ldots, x_1x_n; x_2x_2, x_2x_3, \ldots, x_2x_n;\ldots; x_{n-1}x_{n-1}, x_{n-1}x_n; x_nx_n).$$
To a self-adjoint operator $T=(a_{ij})_{i,j=1}^n$ on $\RR^n$, we associate
a vector $\tilde{T}$ in $\RR^{\frac{n(n+1)}{2}}$ by:
\[ \tilde{T}= (a_{11}, 2a_{12},\ldots, 2a_{1n}; a_{22}, 2a_{23}, \ldots, 2a_{2n}; \ldots; a_{(n-1)(n-1)}, 2a_{(n-1)n}; a_{nn}).\]
\end{definition}

Proposition \ref{prop1} now becomes:

\begin{corollary}\label{Coro2}
Given a self-adjoint operator $T=(a_{ij})_{i,j=1}^n$ on $\RR^n$ and
a vector $x=(x_1,x_2,\ldots,x_n)\in \RR^n$, we have
\[ \langle Tx,x\rangle = \langle \tilde{T},\tilde{x}\rangle.\]
\end{corollary}

We are now able to give a classification of the frames
$\chi$ which give injectivity for the quantum detection problem.

\begin{theorem}\label{finite real}
	Let $\chi=\{x_k\}_{k=1}^m$ be a frame for $\RR^n$.  The following
	are equivalent:
	\begin{enumerate}
		\item $\chi$ gives injectivity.
		\item We have that $\{\tilde{x}_k\}_{k=1}^m$ spans $\mathcal{K}:=\RR^\frac{n(n+1)}{2}$.
	\end{enumerate}
\end{theorem}
\begin{proof}
	$(1)\Rightarrow (2)$: Let a vector
	$$a=(a_{11}, a_{12},\ldots, a_{1n}; a_{22}, a_{23}, \ldots, a_{2n}; \ldots; a_{(n-1)(n-1)}, a_{(n-1)n}; a_{nn})\in \mathcal{K} $$ be such that $\langle a, \tilde{x}_k\rangle=0$ for all $k$.
	
	Define an operator $T=(b_{ij})_{i,j=1}^n$ on $\RR^n$, where 
	$b_{ii}=a_{ii}$ for $i=1, 2, \ldots, n$ and $b_{ij}=b_{ji}=\dfrac{1}{2}a_{ij}$ for $i<j$. Then $T$ is a self-adjoint operator.
	
	For any $x=(x_1, x_2, \ldots, x_n)\in \RR^n$ we have
	
	\begin{align*}
	\langle Tx,x\rangle&= \sum_{i=1}^{n}\sum_{j=1}^{n}b_{ij}x_ix_j\\
	&=\sum_{i=1}^{n}b_{ii}x_i^2+2\sum_{1\leq i<j\leq n}b_{ij}x_ix_j\\
	&=\sum_{i=1}^{n}a_{ii}x_i^2+\sum_{1\leq i<j\leq n}a_{ij}x_ix_j\\
	&=\langle a, \tilde{x}\rangle.
	\end{align*}
	Therefore, $\langle Tx_k,x_k\rangle=\langle a, \tilde{x}_k\rangle=0$ for all $k$. This implies $T=0$  and hence $a=0$.
	
	$(2)\Rightarrow (1)$: Let $T=(a_{ij})_{i,j=1}^n$ be a self-adjoint operator such that
	$$\langle Tx_k, x_k\rangle=0, \mbox{ for all } k.$$
	Then by Corollary \ref{Coro2},
	$$\langle \tilde{T}, \tilde{x}_k\rangle=\langle Tx_k, x_k\rangle=0, \mbox{ for all } k.$$
	Since $\{\tilde{x}_k\}_{k=1}^m$ spans $\mathcal{K}$, we
	have that $\tilde{T}=0$ and so $T=0$.	
\end{proof}

The theorem gives a lower limit on the number of vectors needed
to achieve injectivity.
\begin{corollary}
	If a frame $\mathcal{X}=\{x_k\}_{k=1}^m$ gives injectivity in $\RR^n$, then $m\geq\dfrac{n(n+1)}{2}$.
\end{corollary}

As a consequence (See the related
\cite{CPT}):

\begin{corollary}
Given a frame $\{x_k\}_{k=1}^m$ for $\RR^n$, the following are
equivalent:
\begin{enumerate}
\item The family $\{x_kx_k^*\}_{k=1}^m$ spans the class of self-adjoint operators.
\item The family of vectors $\{\tilde{x}_k\}_{k=1}^m$ spans
$\RR^{\frac{n(n+1)}{2}}$.
\end{enumerate}
\end{corollary}

\begin{proof}
	This is immediate since for every $x\in \RR^n$ and self-adjoint operator $T$, we have
	\[\langle T, xx^*\rangle=\tr(Txx^*)=\langle Tx, x\rangle.\]
\end{proof}

\begin{remark}
For any of the frames $\{x_k\}_{k=1}^m$ giving injectivity, if
$S$ is the frame operator, then $\{S^{-1/2}x_k\}_{k=1}^m$ is a 
Parseval frame giving injectivity by Corollary \ref{cor1}.
\end{remark}

	Normally in the frame quantum detection problem, there is the
	added assumption that the trace of the operators is one.
We will now see that with this assumption, we can eliminate one measurement. We start with a simple example.

\begin{example}
	Let $\mathcal{X}=\{(1,0), (1,1)\}$ in $\RR^2$. Then $\mathcal{X}$ gives injectivity in $\RR^2$ for all self-adjoint operators of trace one.
\end{example}	
Indeed, let 
\[ T=\begin{bmatrix}
a&b\\
b&c
\end{bmatrix}
\] be a self-adjoint matrix of trace zero such that 
$$\langle T(1,0),(1,0)\rangle  = \langle T(1,1),
(1,1)\rangle=0.$$
Then
\[ a= \langle T(1,0),(1,0)\rangle = 0.\] Since $a+c=0$ and 
\[ \langle T(1,1),(1,1)\rangle = \langle (a+b,b+c),
(1,1)\rangle = a+2b+c,\]
then $b=c=0$. So $T=0$.	
	
For the classification of all frames which give injectivity with this added assumption, we will need:

	\begin{definition}
		Let $x=(x_1, x_2, \ldots, x_n)\in \RR^n$. Define
		$$\tilde{x}=(x_1x_2, \ldots, x_1x_n; x_2^2-x_1^2, x_2x_3, \ldots, x_2x_n;\ldots; x_{n-1}^2-x_1^2, x_{n-1}x_n; x_n^2-x_1^2).$$
	\end{definition}
	Now we can prove the trace one version of our classification.

\begin{theorem}
	Let $\mathcal{X}=\{x_k\}_{k=1}^m$ be a frame for $\RR^n$.  The following
	are equivalent:
	\begin{enumerate}
		\item $\mathcal{X}$ gives injectivity for all self-adjoint
		operators of trace one.
		\item We have that $\{\tilde{x}_k\}_{k=1}^m$ spans $\mathcal{K}:=\RR^{\frac{n(n+1)}{2}-1}$.
	\end{enumerate}
\end{theorem}

\begin{proof}
	Note that we are trying to show that when two  positive, self-adjoint operators 
	$T, S$ of trace one satisfy
	\[ \langle Tx_k,x_k\rangle = \langle Sx_k,x_k\rangle,
	\mbox{ for all }k=1,2,\ldots,m,\]
	then $T=S$.  This is clearly equivalent to showing that
	if $T$ is a self-adjoint operator of trace zero and
	$\langle Tx_k,x_k\rangle =0$ for all $k=1,2,\ldots,$ then
	$T=0$.

	\vskip10pt	
	$(1)\Rightarrow (2)$: Let a vector
	$$a=(a_{12}, \ldots, a_{1n}; a_{22}, \ldots, a_{2n}; \ldots; a_{(n-1)(n-1)}, a_{(n-1)n}; a_{nn})\in \mathcal{K}$$ be such that $\langle a, \tilde{x}_k\rangle=0$ for all $k$.
	
	Define an operator $T=(b_{ij})_{i,j=1}^n$, where 
	$b_{11}=-\sum_{i=2}^na_{ii}, b_{ii}=a_{ii}$ for $i=2, 3, \ldots, n$ and $b_{ij}=b_{ji}=\dfrac{1}{2}a_{ij}$ for $i<j$. Then $T$ is self-adjoint and $\tr(T)=0$.
	
	For any $x=(x_1, x_2, \ldots, x_n)\in \RR^n$ we have
	\begin{align*}
	\langle Tx,x\rangle&= \sum_{i=1}^{n}\sum_{j=1}^{n}b_{ij}x_ix_j\\
	&=\sum_{i=1}^{n}b_{ii}x_i^2+2\sum_{1\leq i<j\leq n}b_{ij}x_ix_j\\
	&=\left(-\sum_{i=2}^{n}a_{ii}\right)x_1^2+\sum_{i=2}^{n}a_{ii}x_i^2+\sum_{1\leq i<j\leq n}a_{ij}x_ix_j\\
	&=\langle a, \tilde{x}\rangle.
	\end{align*}
	Therefore, $\langle Tx_k,x_k\rangle=\langle a, \tilde{x}_k\rangle=0$ for all $k$. This implies $T=0$ and hence $a=0$.
	
	$(2)\Rightarrow (1)$: Let $T=(a_{ij})_{i,j=1}^n$ be a self-adjoint operator with $\tr(T)=0$ and such that
	$\langle Tx_k, x_k\rangle=0$ for all $k$. Then $a_{11}=-\sum_{i=2}^{n}a_{ii}.$
	
	Define 
	$$\tilde{T}=(2a_{12}, 2a_{13}, \ldots, 2a_{1n}; a_{22}, 2a_{23}, \ldots, 2a_{2n}; \ldots; a_{(n-1)(n-1)}, 2a_{(n-1)n}; a_{nn}).$$
	Then $\tilde{T}\in \mathcal{K}$ and
	$$\langle \tilde{T}, \tilde{x}_k\rangle=\langle Tx_k, x_k\rangle=0, \mbox{ for all } k.$$
	Since $\{\tilde{x}_k\}_{k=1}^m$ spans $\mathcal{K}$, then $\tilde{T}=0$. Hence $T=0$.	
\end{proof}

\subsubsection{The complex case}
We need to adjust some of definitions for the real case so they
will apply to the complex case.

\begin{definition}\label{defn finite complex}
	Given $x=(x_1,x_2,\ldots,x_n) \in \CC^n$, define
	\begin{align*}
	\tilde{x} = (|x_1|^2,\R(\bar{x}_1x_2),& \I(\bar{x}_1x_2),\ldots,\R(\bar{x}_1x_n),\I(\bar{x}_1x_n);\\ &|x_2|^2,\R(\bar{x}_2x_3),\I(\bar{x}_2x_3),
	\ldots,\R(\bar{x}_2x_n),\I(\bar{x}_2x_n);\ldots;\\
	& |x|^2_{n-1},\R(\bar{x}_{n-1}x_n), \I(\bar{x}_{n-1}x_n); |x_n|^2)\in \RR^{n^2}.
	\end{align*}
\end{definition}

Now we can give our classification theorem for injectivity in
the quantum detection problem for the complex case.

\begin{theorem}
	Let $\mathcal{X}=\{x_k\}_{k=1}^m$ be a frame for $\CC^n$.  The following
	are equivalent:
	\begin{enumerate}
		\item $\mathcal{X}$ gives injectivity.
		\item We have that $\{\tilde{x}_k\}_{k=1}^m$ spans $\RR^{n^2}$.
	\end{enumerate}
\end{theorem}
\begin{proof}
	$(1)\Rightarrow (2)$: 
	Let a be any vector 
	\begin{align*}
	a= (a_{11},u_{12},v_{12},\ldots,u_{1n},v_{1n}; &a_{22},u_{23},v_{23},
	\ldots,u_{2n},v_{2n};\ldots;\\
	&a_{(n-1)(n-1)},u_{(n-1)n},v_{(n-1)n}; a_{nn})\in \RR^{n^2}
	\end{align*}
	such that $\langle a,\tilde{x}_k\rangle =0$ for all $k$.
	
	Define an operator $T=(b_{ij})_{i,j=1}^n$ with 
	$b_{ii}=a_{ii}$ for $i=1, 2, \ldots, n$ and $b_{ij}=\bar{b}_{ji}=\dfrac{1}{2}(u_{ij}-\iota v_{ij})$ for $i<j$. Then $T$ is a self-adjoint operator.
	
	For any $x=(x_1, x_2, \ldots, x_n)\in \CC^n$ we have
	\begin{align*}
	\langle Tx,x\rangle &= \sum_{i=1}^n \sum_{j=1}^n b_{ij}\bar{x}_ix_j\\
	&= \sum_{i=1}^nb_{ii}|x_i|^2 + \sum_{1\leq i<j\leq n}b_{ij}\bar{x}_ix_j
	+ \sum_{1\leq j<i\leq n}b_{ij}\bar{x}_ix_j\\
	&= \sum_{i=1}^nb_{ii}|x_i|^2 +\sum_{1\leq i<j\leq n}b_{ij}\bar{x}_ix_j
	+ \sum_{1\leq j<i\leq n}\bar{b}_{ji}\bar{x}_ix_j\\
	&= \sum_{i=1}^nb_{ii}|x_i|^2 +\sum_{1\leq i<j\leq n}b_{ij}\bar{x}_ix_j
	+ \sum_{1\leq i<j\leq n}\bar{b}_{ij}\bar{x}_jx_i\\
	&=\sum_{i=1}^nb_{ii}|x_i|^2 + 2\sum_{1\leq i<j\leq n}\R( b_{ij}\bar{x}_ix_j)\\
	&=\sum_{i=1}^nb_{ii}|x_i|^2 + 2\sum_{1\leq i<j\leq n}(\R( b_{ij})\R(\bar{x}_ix_j)-\I( b_{ij})\I(\bar{x}_ix_j))\\
	&=\sum_{i=1}^{n}a_{ii}|x_i|^2+\sum_{1\leq i<j\leq n}\left(u_{ij}\R(\bar{x}_ix_j)+v_{ij}\I(\bar{x}_ix_j)\right)\\
	&=\langle a,\tilde{x}\rangle.
	\end{align*}
	Therefore, $\langle Tx_k,x_k\rangle=\langle a, \tilde{x}_k\rangle=0$ for all $k$. This implies $T=0$ and hence $a=0$.
	
	$(2)\Rightarrow (1)$: Let $T=(a_{ij})_{i,j=1}^n$ be a self-adjoint operator such that
	$$\langle Tx_k, x_k\rangle=0, \mbox{ for all } k.$$
	
	Define 
	\begin{align*}
	\tilde{T} = (a_{11},&2\R(a_{12}), -2\I(a_{12}),\ldots,2\R(a_{1n}),-2\I(a_{1n});\\ &a_{22},2\R(a_{23}),-2\I(a_{23}),
	\ldots,2\R(a_{2n}),-2\I(a_{2n});\ldots; \\
	&a_{(n-1)(n-1)},2\R(a_{(n-1)n}), -2\I(a_{(n-1)n}); a_{nn})\in \RR^{n^2}.
	\end{align*}
	Then we have
	$$\langle \tilde{T}, \tilde{x}_k\rangle=\langle Tx_k, x_k\rangle=0, \mbox{ for all } k.$$
	Since $\{\tilde{x}_k\}_{k=1}^m$ spans $\mathcal{K}$ we have that $\tilde{T}=0$ 
	and so $T=0$; i.e. $\{x_k\}_{k=1}^m$ gives injectivity.
	
\end{proof}

\begin{corollary}
	If a frame $\mathcal{X}=\{x_k\}_{k=1}^m$ gives injectivity in $\CC^n$, then $m\geq n^2$.
\end{corollary}
Similar to the real case, we have a classification for injectivity for positive self-adjoint operators of trace one in a complex Hilbert space.  This requires another definition to fit this case.
\begin{definition}
	Given $x=(x_1,x_2,\ldots,x_n) \in \CC^n$, define
	\begin{align*}
	\tilde{x} = (\R(\bar{x}_1x_2),& \I(\bar{x}_1x_2),\ldots,\R(\bar{x}_1x_n),\I(\bar{x}_1x_n);\\ &|x_2|^2-|x_1|^2,\R(\bar{x}_2x_3),\I(\bar{x}_2x_3),
	\ldots,\R(\bar{x}_2x_n),\I(\bar{x}_2x_n);\ldots;\\
	& |x_{n-1}|^2-|x_1|^2,\R(\bar{x}_{n-1}x_n), \I(\bar{x}_{n-1}x_n); |x_n|^2-|x_1|^2)\in \RR^{n^2-1}.
	\end{align*}
\end{definition}
Now we classify the frames which give injectivity in the complex
case for operators of trace one.

\begin{theorem}
	Let $\mathcal{X}=\{x_k\}_{k=1}^m$ be a frame for $\CC^n$.  The following
	are equivalent:
	\begin{enumerate}
		\item $\mathcal{X}$ gives injectivity for all self-adjoint operators of trace one.
		\item We have that $\{\tilde{x}_k\}_{k=1}^m$ spans $\RR^{n^2-1}$.
	\end{enumerate}
\end{theorem}
\begin{proof}
	$(1)\Rightarrow (2)$: 
	Let a vector 
	\begin{align*}
	a= (u_{12},v_{12},\ldots,u_{1n},v_{1n}; &a_{22},u_{23},v_{23},
	\ldots,u_{2n},v_{2n};\ldots;\\
	&a_{(n-1)(n-1)},u_{(n-1)n},v_{(n-1)n}; a_{nn})\in \RR^{n^2-1}
	\end{align*}
	be such that $\langle a,\tilde{x}_k\rangle =0$ for all $k$.
	
	Define an operator $T=(b_{ij})_{i,j=1}^n$ with $b_{11}=-\sum_{i=2}^{n}a_{ii}$,
	$b_{ii}=a_{ii}$ for $i= 2, \ldots, n$ and $b_{ij}=\bar{b}_{ji}=\dfrac{1}{2}(u_{ij}-\iota v_{ij})$ for $i<j$. Then $T$ is a self-adjoint operator and $\tr(T)=0$.
	
	For any $x=(x_1, x_2, \ldots, x_n)\in \CC^n$ we have
	\begin{align*}
	\langle Tx,x\rangle &=\sum_{i=1}^nb_{ii}|x_i|^2 + 2\sum_{1\leq i<j\leq n}(\R( b_{ij})\R(\bar{x}_ix_j)-\I( b_{ij})\I(\bar{x}_ix_j))\\
	&=\left(-\sum_{i=2}^{n}a_{ii}\right)|x_1|^2+\sum_{i=2}^{n}a_{ii}|x_i|^2+\sum_{1\leq i<j\leq n}\left(u_{ij}\R(\bar{x}_ix_j)+v_{ij}\I(\bar{x}_ix_j)\right)\\
	&=\langle a,\tilde{x}\rangle.
	\end{align*}
	Therefore, $\langle Tx_k,x_k\rangle=\langle a, \tilde{x}_k\rangle=0$ for all $k$. This implies $T=0$ and hence $a=0$.
	
	$(2)\Rightarrow (1)$: Let $T=(a_{ij})_{i,j=1}^n$ be a self-adjoint operator such that $\tr(T)=0$ and
	$\langle Tx_k, x_k\rangle=0$ for all $k$. Then $a_{11}=-\sum_{i=2}^{n}a_{ii}.$
	
	Define 
	\begin{align*}
	\tilde{T} = &(2\R(a_{12}), -2\I(a_{12}),\ldots,2\R(a_{1n}),-2\I(a_{1n});\\ &a_{22},2\R(a_{23}),-2\I(a_{23}),
	\ldots,2\R(a_{2n}),-2\I(a_{2n});\ldots; \\
	&a_{(n-1)(n-1)},2\R(a_{(n-1)n}), -2\I(a_{(n-1)n}); a_{nn})\in \RR^{n^2-1}.
	\end{align*}
	Then we have that
	$$\langle \tilde{T}, \tilde{x}_k\rangle=\langle Tx_k, x_k\rangle=0, \mbox{ for all } k.$$
	Since $\{\tilde{x}_k\}_{k=1}^m$ spans $\RR^{n^2-1}$ then $\tilde{T}=0$. Hence $T=0$. 
\end{proof}

Now we will give another classification of injectivity for the
quantum detection problem.  This classification has the disadvantage
that the requirements are quite complex and difficult to verify
in practice.  The advantage here is in the other direction.  That
is, if a frame gives injectivity in the quantum detection problem,
then it must satisfy these complex requirements.

\begin{theorem}
	Let $\mathcal{X}=\{x_k\}_{k=1}^m$ be a frame for a real or complex Hilbert space  $\HH^n$.  The following
	are equivalent:
	\begin{enumerate}
		\item $\mathcal{X}$ gives injectivity.
		\item For every orthonormal basis 
		$\mathcal{E}=\{e_j\}_{j=1}^n$ for $\HH^n$
		we have:
		\[ H(\mathcal{E})=:\spn\{(|\langle x_k,e_1\rangle|^2,|\langle x_k,e_2\rangle|^2,
		\ldots, |\langle x_k,e_n\rangle|^2):k=1,2,\ldots,m\}=\RR^n.\] 
	\end{enumerate}
\end{theorem}

\begin{proof}
	$(1)\Rightarrow (2)$:
	We prove the contrapositive.  Suppose that (2) fails. Then there is an
	orthonormal basis $\mathcal{E}=\{e_j\}_{j=1}^n$ so that
	$H(\mathcal{E})\not= \RR^n$.  Hence there is a non-zero vector
	$\lambda=(\lambda_1,\lambda_2,\ldots,\lambda_n)\in \RR^n$ such that $\lambda\perp H(\mathcal{E})$.
	
	Define an operator on $\HH^n$ by
	\[ Te_j = \lambda_je_j, j=1, 2, \ldots, n.\]
	Then $T$ is a non-zero self-adjoint operator and satisfies 
	$\langle Tx_k,x_k\rangle =0$, for all $k=1, 2,\ldots,m$, which is a contradiction.
	
	$(2)\Rightarrow (1)$:
	Let $T$ be a self-adjoint operator such that  $\langle Tx_k,x_k\rangle =0$, for all $k$.
	Let $\mathcal{E}= \{e_j\}_{j=1}^n$ be an eigenbasis for $T$
	with respective eigenvalues $\{\lambda_j\}_{j=1}^n$.  Then
	for every $k=1,2,\ldots,m$ we have
	\[ \langle Tx_k,x_k\rangle = \sum_{j=1}^n\lambda_j|\langle x_k,
	e_j\rangle|^2 =0.\]
	That is,
	\[ (\lambda_1,\lambda_2,\ldots,\lambda_n) \perp H(\mathcal{E}) = \RR^n \mbox{ by assumption (2)}.\]
	Therefore, $\lambda_j=0$ for all $j=1,2,\ldots,n$
	and so $T=0$.
\end{proof}
 Finally in this subsection, we notice that if a family of vectors gives injectivity in a Hilbert space $\HH^n$, then it is a frame for $\HH^n$.
	\begin{proposition}
		Let $\{x_k\}_{k=1}^m$ be a family of vectors in $\HH^n$ which is injective. Then $\spn\{x_k\}_{k=1}^m=\HH^n$.
	\end{proposition}
	\begin{proof}
		Suppose by contradiction that $W:=\spn\{x_k\}_{k=1}^m\not=\HH^n$. Let $P$ be the orthogonal projection onto $W^\perp$. Then
		$\langle Px_k, x_k\rangle=0$
		for all $k$, but $P\not=0$, a contradiction.
	\end{proof}
\subsection{Constructing the Solutions to the Injectivity Problem}\

In this subsection, we will construct large classes of frames which give injectivity for the quantum detection problem in both the real and complex cases.

\begin{theorem}\label{thm2}
	Let $\{x_k\}_{k=1}^n$ be a linearly independent set in $\RR^n$ such that the first coordinates of these vectors are non-zero.
	Now choose $(n-1)$ linearly independent vectors $\{x_k\}_{k=n+1}^{2n-1}$ in $\RR^n$ such that each vector is zero in the first coordinate and is non-zero in the second coordinate.
	Continuing this procedure we get a frame $\{x_k\}_{k=1}^{\frac{n(n+1)}{2}}$ which gives injectivity.
\end{theorem}
\begin{proof}
	We will show that $\{\tilde{x}_k\}_{k=1}^{\frac{n(n+1)}{2}}$ is a basis for $\RR^{\frac{n(n+1)}{2}}$.
	
	Indeed, suppose that $\sum_{k=1}^{\frac{n(n+1)}{2}}\alpha_k\tilde{x}_k=0$ for some scalars $\{\alpha_k\}$. Since after $n$, all tilde vectors are zero in the first coordinate, then we get 
	
	$$\sum_{k=1}^n\alpha_kx_{k1}x_k=0.$$  Since $\{x_k\}_{k=1}^n$ are linearly independent, $\alpha_kx_{k1}=0$ for all $k$ and
	since $x_{k1}\not=0$, $ \alpha_k=0$ for $k=1,2,\ldots,n.  $
	
	Now do this argument for the next $(n-1)$ vectors and continue we get $\alpha_k=0$ for all $k=1, 2, \ldots, \frac{n(n+1)}{2}$.
\end{proof}
A simple example satisfying the construction is the following.
\begin{ex}
 The frame
 $$\{e_i\}_{i=1}^n\cup\{e_i+e_j : i<j\}_{i,j=1}^n$$ gives injectivity.
\end{ex}

For the complex case, we have the following construction. The proof is as in the real case.

\begin{theorem}
	Let $\{x_k\}_{k=1}^{2n-1}$ be a basis for $\RR^{2n-1}$, where
	$$x_k=(u_{k1}, u_{k2}, v_{k2}, \ldots, u_{kn}, v_{kn})$$ and $u_{k1}\not=0$, $k=1,\ldots, 2n-1.$
	
	Define $(2n-1)$ vectors $\{z_k\}_{k=1}^{2n-1}$ in $\CC^n$ by
	$$z_k=(u_{k1}, u_{k2}+\iota v_{k2}, \ldots, u_{kn}+\iota v_{kn}).$$
	
	Now let $\{x_k\}_{k=2n}^{4n-4}$ be a basis for $\RR^{2n-3}$, where 
	$$x_k=(u_{k2}, u_{k3}, v_{k3}, \ldots, u_{kn}, v_{kn})$$ and $u_{k2}\not=0$, $k=2n, \ldots, 4n-4$.
	
	Define $(2n-3)$ vectors $\{z_k\}_{k=2n}^{4n-4}$ in $\CC^n$ by
	$$z_k=(0, u_{k2}, u_{k3}+\iota v_{k3}, \ldots, u_{kn}+\iota v_{kn}).$$
	
	Continuing this procedure we get $n^2$ vectors $\{z_k\}_{k=1}^{n^2}$ in $\CC^n$ and they give injectivity.
	
\end{theorem}

As we have seen, we can get Parseval frames giving injectivity
by taking $\{S^{-1/2}x_k\}_{k=1}^m$, where $\{x_k\}_{k=1}^m$
gives injectivity and has frame operator $S$.  But the above
construction can be adjusted to directly construct Parseval
frames giving injectivity.

\begin{theorem}
	Let $\{\lambda_{ij}\}_{i=1,j=i}^{\ n\ \ \ n}$ be non-negative numbers
	satisfying:
	\begin{enumerate}
		\item 
		$\lambda_{ij}=0\mbox{ if and only if } j<i.$
		\item For each $j=1,2,\ldots,n$ we have
		$ \sum_{i=1}^n \lambda_{ij}=1.$
	\end{enumerate}
	Let $\mathcal{E}=\{e_j\}_{j=1}^n$ be the canonical basis of $\RR^n$.
	Let $\{x_k\}_{k=1}^{\frac{n(n+1)}{2}}$ be vectors in $\RR^n$
	which satisfy:
	\begin{enumerate}
		\item $\{x_k\}_{k=1}^n$ is a linearly independent set with
		$x_{k1}\not= 0$ for all $k=1,\ldots,n$ and it
		has frame operator $S_1$ with eigenvectors $\mathcal{E}$ and
		respective eigenvalues $\{\lambda_{1j}\}_{j=1}^n$
		(See \cite{C}.)
		\item $\{x_k\}_{k=n+1}^{2n-1}$ is a linearly independent set
		with $x_{k1}=0$, for all $k$, $x_{k2}\not= 0$ for all $k$, and it
		has frame operator $S_2$ with eigenvectors $\mathcal{E}$ and
		respective eigenvalues $\{\lambda_{2j}\}_{j=1}^n$.
		\item continue.
	\end{enumerate}
	Then the vectors $\{x_k\}_{k=1}^{\frac{n(n+1)}{2}}$ form a Parseval
	frame for $\RR^n$ which is injective.
\end{theorem}

\begin{proof}
	This is injective by Theorem \ref{thm2}. To see that it is
	Parseval, observe that the frame operator of this frame is
	$\sum_{i=1}^nS_i$.  Now, let $y\in \RR^n$ and compute:
	\begin{align*}
	\sum_{i=1}^nS_iy&= \sum_{i=1}^n S_i\left(\sum_{j=1}^n \langle y,e_j\rangle e_j\right)\\
	&= \sum_{i=1}^n\sum_{j=1}^n \langle y,e_j\rangle S_ie_j\\
	&= \sum_{i=1}^n \sum_{j=1}^n \langle y,e_j\rangle \lambda_{ij}e_j\\
	&= \sum_{j=1}^n \langle y, e_j\rangle e_j \sum_{i=1}^n \lambda_{ij}\\
	&= \sum_{j=1}^n \langle y,e_j\rangle e_j
	= y.
	\end{align*}   
\end{proof}
Similarly, we have the following theorem for the complex case.
\begin{theorem}
	Fix $\{\lambda_{ij}\}_{i=1,j=i}^{\\ n\ n}$ be non-negative numbers
	satisfying:
	\begin{enumerate}
		\item 
		$\lambda_{ij}=0\mbox{ if and only if } j<i.$
		\item For each $j=1,2,\ldots,n$ we have
		$ \sum_{i=1}^n \lambda_{ij}=1.$
	\end{enumerate}
	Let $\mathcal{E}=\{e_j\}_{j=1}^n$ be the canonical basis of $\CC^n$.
	Let $\{z_k\}_{k=1}^{n^2}$ be vectors in $\CC^n$
	which satisfy:
	\begin{enumerate}
		\item For each $k=1, \ldots, 2n-1$, $z_k$ has the form 
		$$z_k=(u_{k1}, u_{k2}+\iota v_{k2}, \ldots, u_{kn}+\iota v_{kn}), $$ where $u_{k1}\not=0$ and the set 
		$\{(u_{k1}, u_{k2}, v_{k2}, \ldots, u_{kn}, v_{kn})\}_{k=1}^{2n-1}$ is linearly independent in $\RR^{2n-1}$. Moreover $\{z_k\}_{k=1}^{2n-1}$ 
		has frame operator $S_1$ with eigenvectors $\mathcal{E}$ and
		respective eigenvalues $\{\lambda_{1j}\}_{j=1}^n$. 
		\item For each $k=2n, \ldots, 4n-4$, $z_k$ has the form 
		$$z_k=(0, u_{k2}, u_{k3}+\iota v_{k3}, \ldots, u_{kn}+\iota v_{kn}),$$ where $u_{k2}\not=0$
		and the set 
		$\{(u_{k2}, u_{k3}, v_{k3}, \ldots, u_{kn}, v_{kn})\}_{k=2n}^{4n-4}$ is linearly independent in $\RR^{2n-3}$. Moreover $\{z_k\}_{k=2n}^{4n-4}$ 
		has frame operator $S_2$ with eigenvectors $\mathcal{E}$ and
		respective eigenvalues $\{\lambda_{2j}\}_{j=1}^n$. 
		\item continue.
	\end{enumerate}
	Then the vectors $\{z_k\}_{k=1}^{n^2}$ form a Parseval
	frame for $\CC^n$ which is injective.
\end{theorem}

\begin{remark}
We can easily vary the above construction to find frames which give injectivity and have any previously prescribed
	eigenvalues for their frame operators.
\end{remark}

We recall:

\begin{definition}
Two orthonormal bases $\{x_k\}_{k=1}^n$ and $\{y_k\}_{k=1}^n$
are {\bf mutually unbiased} if 
\[ |\langle x_k,y_j\rangle|= \frac{1}{\sqrt{n}},\mbox{ for all }
i,j=1,2,\ldots,n.\]
A family of orthonormal bases is {\bf mutually unbiased} if
each pair is mutually unbiased.
\end{definition}

It is known that the maximal number of mutually unbiased bases
in $\HH^n$ is n+1 and this is rarely achieved.  It holds if
$n=p^m$ for a prime p.  It is observed in \cite{S} and 
\cite{BK} that a maximal family of mutually unbiased bases will
give injectivity in the quantum detection problem.

\subsection{The Solutions are Open and Dense}

In this section we will show that the family of
$m$-element frames which solve the quantum
detection injectivity problem is open and dense in the family
of all $m$-element frames.  For this, we need to measure the
distance between $m$-element frames.
There is a standard metric measuring the distance between
frames.

\begin{definition}
Given frames $\mathcal{X}=\{x_k\}_{k=1}^m$ and $\mathcal{Y}=\{y_k\}_{k=1}^m$ for a 
real or complex Hilbert space $\HH^n$, the {\bf distance} between
them is
\[ d(\mathcal{X},\mathcal{Y})^2 = \sum_{k=1}^m\|x_k-y_k\|^2.\]
\end{definition}

\begin{theorem}
		The set of all $m$-element frames on $\HH^n$ that give injectivity in the
		frame quantum detection problem is dense in the space of all 
		$m$-element frames on $\HH^n$.
	\end{theorem}
	\begin{proof}
		We will prove the real case. The complex case is similar.
		
		Let a frame $\{x_k\}_{k=1}^{\frac{n(n+1)}{2}}\subset \RR^n$ give injectivity. By Theorem \ref{finite real}, this is equivalent to the determinant of the matrix whose rows are $\tilde{x}_k$, $k=1, 2, \ldots, \frac{n(n+1)}{2}$ being non-zero.
		
		The determinant of this matrix is a polynomial of $\frac{n^2(n+1)}{2}$ variables $x_{ki}$ for $1 \leq k \leq \frac{n(n+1)}{2}$ and $1 \leq i \leq n$. Since the complement of the zero set of this polynomial is dense in $\RR^{\frac{n^2(n+1)}{2}}$, the set of all $\frac{n(n+1)}{2}$-element frames which give injectivity is dense in the space of all $\frac{n(n+1)}{2}$-element frames on $\RR^n$.
		
		Now let any $m$-element frame $\{x_k\}_{k=1}^m$ in $\RR^n$ with $m\geq \frac{n(n+1)}{2}$ and $\delta >0$. Then there exists a subframe containing $\frac{n(n+1)}{2}$ vectors. We can assume that this subframe is  $\{x_k\}_{k=1}^{\frac{n(n+1)}{2}}$. By denseness above, there is an injective frame $\{y_k\}_{k=1}^{\frac{n(n+1)}{2}}$ such that

		\[ \sum_{k=1}^{n(n+1)/2}\|x_k-y_k\|^2<\delta.\]

		Now define a new frame $\{\phi_k\}_{k=1}^m$, where $\phi_k=y_k$ for $k=1, \ldots, \frac{n(n+1)}{2}$ and $\phi_k=x_k$ for $k>\frac{n(n+1)}{2}$.
		Then the frame $\{\phi_k\}_{k=1}^{m}$ is injective and 
		
		\[ \sum_{k=1}^m\|x_k-\phi_k\|^2<\delta.\]
		
		The conclusion of the theorem then follows.
	\end{proof}
	\begin{remark}
		In the real case it is known that the complement of the zero set of a nontrivial polynomial of $n$ variables is dense in $\RR^n$. In the complex case, we see that given a polynomial $P(z_1,..., z_n)$ on $\CC^n$, we may write $P$ as $$P'(x_1,y_1,...,x_n,y_n) + \iota P''(x_1,y_1,...,x_n,y_n)$$ where $z_j = x_j + \iota y_j$. Hence $P'$ and $P''$ are polynomials on $\RR^{2n}$. $P$ has a zero if and only if $P'$ and $P''$ have a common zero. We see that the complement of the intersection of the zero sets of $P'$ and $P''$ is dense in $\RR^{2n}$ and hence is dense in $\CC^{n}$ after natural identification of $\RR^{2n}$ with $\CC^n$.
	\end{remark}
	
	\begin{theorem}
		The family of all $m$-element frames on $\HH^n$ that give injectivity in the
		frame quantum detection problem is open in the space of all $m$-element
		frames on $\HH^n$.
	\end{theorem}
	\begin{proof}
		As above we will prove the real case and the complex
		case follows similarly.
		
		Denote by $\mathcal{F}$ the space of all $\frac{n(n+1)}{2}$-element frames for $\RR^n$. Consider the map:
		\begin{align*} f: \mathcal{F} &\longrightarrow \RR\\
		\mathcal{X}=\{x_k\}_{k=1}^{\frac{n(n+1)}{2}}&\longmapsto  f(\mathcal{X})=\det\{\tilde{x}_1, \tilde{x}_2,\ldots, \tilde{x}_{\frac{n(n+2)}{2}}\}.
		\end{align*}
		Then $f$ is a continuous function. Since $f^{-1}(0)$ is a closed set, by Theorem \ref{finite real}, the set of all $\frac{n(n+1)}{2}$-element frames is open in $\mathcal{F}$.
		
		Now let$\mathcal{X}=\{x_k\}_{k =1}^m$ in $\RR^n, (m\geq \frac{n(n+1)}{2})$ be an m-element frame which gives injectivity. Then there is a subframe $\mathcal{Y}$ containing $\frac{n(n+1)}{2}$ vectors, which is also injective. Therefore, there exists $\epsilon>0$ such that every $\frac{n(n+1)}{2}$-element frame in the ball $B(\mathcal{Y}, \epsilon)$ is injective. This implies that every $m$-element frame in the ball $B(\mathcal{X}, \epsilon)$ is also injective. The proof is now complete.
	\end{proof}

To show that the Parseval frames giving injectivity in the
quantum detection problem are dense in the Parseval frames,
we will first prove a very general problem about frames.

\begin{theorem}
Let $\mathcal{P}$ be a property of Hilbert space frames and assume:
\begin{enumerate}
\item The set of all $m$-element frames in $\HH^n$ having property $\mathcal{P}$
is dense in the set of all $m$-element frames.
\item If a frame $\{x_k\}_{k=1}^m$ with frame operator $S$ has property $\mathcal{P}$, then $\{S^{-1/2}x_k\}_{k=1}^m$ has 
property $\mathcal{P}$.
\end{enumerate}
Then the set of all $m$-element Parseval frames with property $\mathcal{P}$ is
dense in the set of all $m$-element Parseval frames. 
\end{theorem}

\begin{proof}
Fix $\epsilon >0$ and let $\delta>0$ so that
\[ 2m\delta^2 +8(m\delta)^2 m(1+\delta)^2<\epsilon, \ 2m\delta <1.\]
  Let
$\{x_k\}_{k=1}^m$ be any Parseval frame for $\HH^n$.  By denseness,
we can choose a frame $\{y_k\}_{k=1}^m$ having property $\mathcal{P}$
and satisfying $\|x_k-y_k\| \le \delta$, for all $k=1,2,\ldots,m$.
Since $\|x_k\|\le 1$, we have that $\|y_k\|\le 1+\delta$.
Let $S_1$ be the frame operator of $\{y_k\}_{k=1}^m$.  Then,
\begin{align*}
\langle S_1x,x\rangle^{1/2} &= \left ( \sum_{k=1}^m|\langle x,y_k\rangle|^2\right )^{1/2}\\
&\leq \left ( \sum_{k=1}^m|\langle x,x_k\rangle|^2\right )^{1/2}
+ \left ( \sum_{k=1}^m|\langle x,x_k-y_k\rangle|^2 \right )^{1/2}\\
&\leq \|x\|+ \|x\|\left (\sum_{k=1}^m\|x_k-y_k\|^2 \right )^{1/2}\\
&\leq \|x\|(1+m\delta).
\end{align*}
Therefore
\[ \sum_{k=1}^m|\langle x,y_k\rangle|^2 \le \|x\|^2(1+m\delta)^2.\]
Similarly,
\[ \sum_{k=1}^m|\langle x,y_k\rangle|^2 \ge \|x\|^2(1-m\delta)^2.\]
I.e. $(1-m\delta)^2 I \le S_1 \le (1+m\delta)^2 I$.  
Hence, $(1-m\delta)I \le S_1^{1/2}\le (1+m\delta)I$ and so
$(1+m\delta)^{-1}I \le S_1^{-1/2}\le (1-m\delta)^{-1}I$.
Finally,
\[ I-(1-m\delta)^{-1}I \le I-S_1^{-1/2} \le I-(1+m\delta)^{-1}I,\]
and so
\[ -2m\delta I \le \frac{-m\delta}{1-m\delta}I
\le I-S_1^{-1/2} \le \frac{m\delta}{1+m\delta}I \le 2m\delta I.\]
 Now, 
$\{S_1^{-1/2}y_k\}_{k=1}^m$ is a Parseval frame with property
$\mathcal{P}$ and
\begin{align*}
\sum_{k=1}^m\|x_k-S_1^{-1/2}y_k\|^2 &\le 2 \sum_{k=1}^m \|x_k-y_k\|^2+
2\sum_{k=1}^m\|(I-S_1^{-1/2})y_k\|^2\\
&\le 2m\delta^2 + 2\sum_{k=1}^m (2m\delta)^2 \|y_k\|^2\\\\
&\le 2m\delta^2 +8(m\delta)^2 m(1+\delta)^2< \epsilon.
\end{align*}  
\end{proof}

\begin{corollary}
The set of all $m$-element Parseval frames which give injectivity is dense in
the set of all $m$-element Parseval frames.
\end{corollary}

\subsection{Solution to the State Estimation Problem}

In this section we will give a classification of injective Parseval frames for which the state estimation problem is solvable. 

Recall that for an injective Parseval frame $\{x_k\}_{k\in I}$ and $\beta=2^I$, the map $\mathbb{M}$ which maps a quantum state $T\in \mathcal{S}(\HH)$ to a function $p\in L(\beta, \RR)$ is injective. Given a function $p\in L(\beta, \RR)$, if $p=\mathbb{M}(T)$ for some $T\in \mathcal{S}(\HH)$, then for any $U\in \beta$, we must have
\[p(U)=\mathbb{M}(T)(U)=\sum_{k\in U}\langle Tx_k, x_k\rangle=\sum_{k\in U}p(\{k\}).\]
Thus, $p$ must be additive and is determined by its value at the singleton sets $\{k\}$ for all $k\in I$. Therefore, for the state estimation problem in the finite case, we will ask:

\vskip10pt
\noindent {\bf The State Estimation Problem:} Given an injective Parseval frames $\{x_k\}_{k=1}^m$
	on $\HH^n$ and a measurement vector $a=(a_1,a_2,\ldots,a_m)\in \RR^m$, 
can we find a positive
self-adjoint trace one operator $T$ so that
\[\langle Tx_k,x_k\rangle= a_k, \mbox{ for all }k?\]

\begin{remark}
	We will not require the operator $T$ of the problem to be positive and trace one. This will be considered as a special case of the problem. Hence, we will say that  the state estimation problem is solvable if there exists a self-adjoint operator $T$ so that 
	\[ \langle Tx_k,x_k\rangle= a_k, \mbox{ for all }k.\]
\end{remark}


We will give a complete classification of injective Parseval frames for which the state estimation problem is solvable. Recall that for a vector $x\in \RR^n$, the vector $\tilde{x}$ is 
defined as in the Definition \ref{def real finite}.

\begin{theorem}\label{state finite}
	Let $\mathcal{X}=\{x_k\}_{k=1}^m$ be an injective Parseval frame for $\RR^n$, and $a=(a_1, a_2, \ldots, a_m)\in \RR^m$, the following are equivalent:
	\begin{enumerate}
		\item The state estimation problem is solvable.
		\item $\rk(A)=\rk(B)$, where $A$ is a matrix
		whose the $k$-row is $\tilde{x}_k$,
		and $B=[A, a]$. 	
	\end{enumerate}
\end{theorem}
\begin{proof}
	Note that a self-adjoint operator $T$ is determined by the values $\langle Te_i, e_j\rangle$ for all $i\leq j$. Then the state estimation problem is solvable if and only if there exists a self-adjoint operator $T$ so that  
	\begin{align*} a_k&=\langle Tx_k,x_k\rangle\\
	&=\langle T(\sum_{i=1}^{n}\langle x_k, e_i\rangle e_i),\sum_{j=1}^{n}\langle x_k, e_j\rangle e_j\rangle\\ &=\sum_{i=1}^{n}\sum_{j=1}^{n}\langle x_k, e_i\rangle\langle x_k, e_j\rangle\langle Te_i, e_j\rangle\end{align*} for all  $k$.
	This is equivalent to the linear system with unknowns $\langle Te_i, e_j\rangle$:
	$$\sum_{i=1}^nx_{ki}^2\langle Te_i, e_i\rangle+2\sum_{i<j}x_{ki}x_{kj}\langle Te_i, e_j\rangle=a_k,\ k=1, 2, \ldots, m$$ having a solution, and hence is equivalent to $\rk(A)=\rk(B)$.
\end{proof}

In the case where the number of frame vectors equals $\frac{n(n+1)}{2}$, we have the following corollary.
\begin{corollary}
	Let $\mathcal{X}=\{x_k\}_{k=1}^{\frac{n(n+1)}{2}}\subset \RR^n$ be an injective Parseval frame. Then the state estimation problem has a unique solution for all choices of vectors $a=(a_1, a_2, \ldots, a_{\frac{n(n+1)}{2}}).$
\end{corollary}
\begin{proof}
	By Theorem \ref{finite real}, $\mathcal{X}$ is injective is equivalent to $\{\tilde{x}_k\}_{k=1}^{\frac{n(n+1)}{2}}$ is linearly independent. Hence 
	$$\rk A=\rk B=\frac{n(n+1)}{2}.$$
	The conclusion then follows by Theorem \ref{state finite}.
\end{proof}
For the completeness of the state estimation problem, we will state the classification in the case that the operator $T$ is required to be positive, self-adjoint operator of trace one. First, we need to recall the following theorem. 
\begin{theorem}
	A self-adjoint matrix $T$ is positive if and only if all of its principal minors are nonnegative.
\end{theorem}
Now we have the following classification: 
\begin{theorem}
	Let $\mathcal{X}=\{x_k\}_{k=1}^m$ be an injective Parseval frame for $\RR^n$, and $a=(a_1, a_2, \ldots, a_m)\in \RR^m$, the following are equivalent:
	\begin{enumerate}
		\item The state estimation problem is solvable for a positive, self-adjoint operator of trace one.
		\item The linear system 
		$$\sum_{i=1}^nx_{ki}^2\langle Te_i, e_i\rangle+2\sum_{i<j}x_{ki}x_{kj}\langle Te_i, e_j\rangle=a_k,\ k=1, 2, \ldots, m$$ has a solution $\{\langle Te_i, e_j\rangle: i\le j\}$, which determines a self-adjoint matrix $T$ such that all of its principal minors are nonnegative, and $\sum_{i=1}^n\langle Te_i, e_i\rangle=1$.	
	\end{enumerate}
\end{theorem}

\begin{remark}
	All of the theorems above still hold for the complex case with the corresponding $\tilde{x}_k$, defined as in Definition \ref{defn finite complex}. We state one of them here, the other are similar to the real case. 
\end{remark}
\begin{theorem}
	Let $\mathcal{X}=\{x_k\}_{k=1}^m$ be an injective Parseval frame for $\CC^n$, and $a=(a_1, a_2, \ldots, a_m)\in \RR^m$, the following are equivalent:
	\begin{enumerate}
		\item The state estimation problem is solvable.
		\item $\rk(A)=\rk(B)$, where $A$ is a matrix
		whose the $k$-row is $\tilde{x}_k$,
		and $B=[A, a]$. 	
	\end{enumerate}
\end{theorem}
\begin{proof}
	In the complex case, a self-adjoint operator $T$ is determined by the values of the real part and imaginary part of $\langle Te_j, e_i\rangle$ for all $i\leq j$. Then the state estimation problem is solvable if and only if there exists a self-adjoint operator $T$ so that  
	\begin{align*} a_k&=\langle Tx_k,x_k\rangle\\
	&=\langle T(\sum_{i=1}^{n}\langle x_k, e_i\rangle e_i),\sum_{j=1}^{n}\langle x_k, e_j\rangle e_j\rangle\\ &=\sum_{i=1}^{n}\sum_{j=1}^{n}\langle x_k, e_i\rangle\overline{\langle x_k, e_j\rangle}\langle Te_i, e_j\rangle\\
	&=\sum_{i=1}^n|x_{ki}|^2\langle Te_i, e_i\rangle+2\sum_{i<j}[\R(\bar{x}_{ki}x_{kj})\R\langle Te_j, e_i\rangle-\I(\bar{x}_{ki}x_{kj})\I\langle Te_j, e_i\rangle]\end{align*} for all  $k$.

	This is equivalent to the following linear system:
	$$\sum_{i=1}^n|x_{ki}|^2\langle Te_i, e_i\rangle+2\sum_{i<j}[\R(\bar{x}_{ki}x_{kj})\R\langle Te_j, e_i\rangle-\I(\bar{x}_{ki}x_{kj})\I\langle Te_j, e_i\rangle]=a_k,$$ $k=1, 2, \ldots, m$ with unknowns $\R\langle Te_j, e_i\rangle,  \I\langle Te_j, e_i\rangle, i\leq j$ having a solution, and hence is equivalent to $\rk(A)=\rk(B)$.
\end{proof}

\begin{remark}
If a frame $\{x_k\}_{k=1}^m$ has $m> \frac{n(n+1)}{2}$ in the
real case, or $m >n^2$ in the complex case, because of redundancy,
it is unlikely the state estimation is solvable.  However, in this
case there is a natural way to find the best estimate for the
problem.  We consider the real case. Note that there always exists a subset $I\subset \{1, 2, \ldots, m\}$ of size $\frac{n(n+1)}{2}$, and a self-adjoint operator $T$  so that $\langle Tx_k, x_k\rangle=a_k$, for all $k\in I$. Therefore, if the state estimation problem is not solvable, it is natural to find such $T$ so that the the distance to the measurement vector $a$:
$$\sum_{k=1}^{m}|\langle Tx_k, x_k\rangle-a_k|^2$$ is minimum. 

To do this,
let $\mathcal{S}$ be the set of all bases of $\RR^{\frac{n(n+1)}{2}}$ that are subsets of $\{\tilde{x}_k\}_{k=1}^m.$ This set is obviously finite. Since each element $\{\tilde{x}_k\}_{k\in I}$ in $\mathcal{S}$ determines a unique self-adjoint operator $T$ satisfying $\langle Tx_k, x_k\rangle=a_k$, for all $k\in I$, we can find the quantum state $T$ that gives the best approximation to the measurement vector $a$ by choosing the set which minimizes the
distance above.

\end{remark}

\section{The Solution for the Infinite Dimensional Case}

 In infinite dimensions we will work with 
 both the trace class operators and the Hilbert Schmidt
 operators.  I.e.  Operators $T=(a_{ij})_{i,j=1}^{\infty}$ with
 $\sum_{i,j=1}^{\infty}|a_{ij}|^2 < \infty.$  This class contains
 the trace class operators.
As in the finite case, we will solve the following frame injectitivity problem:
\vskip10pt
\noindent {\bf Injectivity Problem}:  For what frames
$\{x_k\}_{k=1}^{\infty}$ in real or complex infinite dimensional Hilbert space $\HH$ do we
have the property:  Whenever $T, S$ are Hilbert Schmidt positive self-adjoint 
operators 
on $\HH$ and $\langle Tx_k,x_k\rangle =\langle Sx_k,x_k\rangle$, for all $k=1,2,\ldots$,
then $T=S$.
\vskip10pt

\begin{remark}
	We will not require our operators to be trace class and trace one. These requirements will be considered as a special case of our problem.
\end{remark}

\subsection{The Solution to the Injectivity Problem}
 In this subsection we will solve the injectivity problem for infinite
dimensional Hilbert spaces.  
Similar to the finite case, we first show that we only need to work with self adjoint operators. Note that the proof ``(1) implies (2)" of Theorem \ref{positivefinite} is not true for the infinite case. So we will give another proof here. The other implications are as in the finite case.
	
	\begin{theorem}\label{TH}
		Given a family of vectors $\mathcal{X}=\{x_k\}_{k=1}^\infty$ in a real or complex Hilbert space $\HH$, the following are equivalent:
		\begin{enumerate}
			\item Whenever $T, S$ are Hilbert Schmidt, positive and self-adjoint, and
			\[ \langle Tx_k,x_k\rangle = \langle Sx_k,x_k\rangle,
			\mbox{ for all k},\]
			then $T=S$.
			\item Whenever $T, S$ are Hilbert Schmidt self-adjoint, and
			\[ \langle Tx_k,x_k\rangle = \langle Sx_k,x_k\rangle,
			\mbox{ for all }k,\]
			then $T=S$.
			\item $\mathcal{X}$ is injective.
		\end{enumerate}
	\end{theorem}
	\begin{proof}
		We will show that (1) implies (2). Let $T, S$ be Hilbert Schmidt self-adjoint operators such that 	\[ \langle Tx_k,x_k\rangle = \langle Sx_k,x_k\rangle,
		\mbox{ for all }k.\] Set $R=T-S$. Then $R$ is also a Hilbert Schmidt  self-adjoint operator. Let $\{e_j\}_{j=1}^\infty$ be an orthonormal basis for $\HH$ and let $\{u_j\}_{j=1}^\infty$ be an eigenbasis for $R$
		with respective eigenvalues $\{\lambda_j\}_{j=1}^\infty$. Define operators $U$ and $D$ on $\HH$ by $Ue_j=u_j$ and $De_j=\lambda_je_j$, for $j=1, 2,\ldots$. Then $U$ is a unitary operator,  $D$ is Hilbert Schmidt self-adjoint operator, and
		$$R=UDU^*.$$ Now let $r_j=|\lambda_j|, s_j=|\lambda_j|-\lambda_j, j=1, 2, \ldots$ 
be non-negative numbers.  Then $\lambda_j=r_j-s_j$. Let $D_1, D_2$ be operators defined by 
		$$D_1e_j=r_je_j, \ D_2e_j=s_je_j \mbox{  for } j=1, 2, \ldots.$$ Note that since $R$ is Hilbert Schmidt, $\sum_{j=1}^{\infty}\lambda_j^2$ converges. Hence
		$D_1, D_2$ are Hilbert-Schmidt positive self-adjoint and we have
		$$R=UDU^*=U(D_1-D_2)U^*=UD_1U^*-UD_2U^*.$$
		Moreover, $UD_1U^*,\ UD_2U^*$ are Hilbert Schmidt positive self-adjoint operators.
		Since 
		$$0=\langle Rx_k, x_k\rangle=\langle UD_1U^*x_k, x_k\rangle-\langle UD_2U^*x_k, x_k\rangle,$$ 
		we have that $UD_1U^*=UD_2U^*$. Thus, $R=0$ and hence $T=S$.
	\end{proof}
	If our operators are trace class, then we will have the following theorem. The proof of  Theorem \ref{TH} is still valid here by noticing that $\sum_{j=1}^{\infty}|\lambda_j|<\infty$ for the trace class operator $R$.
	\begin{theorem}
		Given a family of vectors $\mathcal{X}=\{x_k\}_{k=1}^\infty$ in a infinite dimensional Hilbert space $\HH$, the following are equivalent:
		\begin{enumerate}
			\item Whenever $T, S$ are trace class positive and self-adjoint, and
			\[ \langle Tx_k,x_k\rangle = \langle Sx_k,x_k\rangle,
			\mbox{ for all k},\]
			then $T=S$.
			\item Whenever $T, S$ are trace class self-adjoint, and
			\[ \langle Tx_k,x_k\rangle = \langle Sx_k,x_k\rangle,
			\mbox{ for all k},\]
			then $T=S$.
			\item Whenever $T$ is trace class self-adjoint, and
			\[ \langle Tx_k,x_k\rangle = 0,
			\mbox{ for all k},\]
			then $T=0$.
		\end{enumerate}
	\end{theorem}

	Similar to the finite case, we will first give the following classification of injectivity for Hilbert Schmidt operators. 
	
	\begin{theorem}
		Let $\mathcal{X}=\{x_k\}_{k=1}^\infty$ be a frame for an infinite dimensional real or complex Hilbert space  $\HH$.  The following
		are equivalent:
		\begin{enumerate}
			\item $\mathcal{X}$ is injective. 
			\item For every orthonormal basis 
			$\mathcal{E}=\{e_j\}_{j=1}^\infty$ for $\HH$
			we have:
			\[ H(\mathcal{E})=:\overline{\spn}\{(|\langle x_k,e_1\rangle|^2,|\langle x_k,e_2\rangle|^2,
			\ldots ):k=1,2,\ldots\}=\ell_2.\] 
		\end{enumerate}
	\end{theorem}
	\begin{proof}
		$(1)\Rightarrow (2)$:
		We prove the result by way of contradiction.  Suppose that (2) is false. Then there is an
		orthonormal basis $\mathcal{E}=\{e_j\}_{j=1}^\infty$ so that
		$H(\mathcal{E})\not= \ell_2$.  Hence there is a non-zero vector
		$\lambda=(\lambda_1,\lambda_2,\ldots)\in \ell_2$ such that $\lambda\perp H(\mathcal{E})$.
		
		Define an operator on $\HH$ by
		\[ Te_j = \lambda_je_j, \mbox{ for all }j=1, 2, \ldots.\]
		Then $T$ is a non-zero Hilbert Schmidt operator.  We also have:  
		$\langle Tx_k,x_k\rangle =0$, for all $k=1, 2,\ldots$. This is a contradiction.
		
		$(2)\Rightarrow (1)$:
		Let $T$ be a Hilbert Schmidt self-adjoint operator such that  $$\langle Tx_k,x_k\rangle =0, \mbox{ for all }k.$$
		Since $T$ is Hilbert Schmidt and hence compact, there
		is an eigenbasis $\mathcal{E}= \{e_j\}_{j=1}^\infty$ for $T$
		with respective eigenvalues $\{\lambda_j\}_{j=1}^\infty$.  Then
		for every $k=1,2,\ldots$, we have
		\[ \langle Tx_k,x_k\rangle = \sum_{j=1}^\infty\lambda_j|\langle x_k,
		e_j\rangle|^2 =0.\]
		
		Since $T$ is Hilbert Schmidt then 
		$$\sum_{j=1}^{\infty}|\lambda_j|^2=\sum_{j=1}^{\infty}\|Te_j\|^2<\infty.$$
		That is, $(\lambda_1, \lambda_2, \ldots)\in \ell_2$. 
		Since \[ (\lambda_1,\lambda_2,\ldots) \perp H(\mathcal{E}) = \ell_2 \mbox{ by assumption (2)}.\]
		Therefore, $\lambda_j=0$ for all $j=1,2,\ldots,$
		and so $T=0$.
	\end{proof}

	If we consider operators which are trace class, then we have the following classification for the infinite dimensions.
	
	\begin{theorem}
		Let $\mathcal{X}=\{x_k\}_{k=1}^\infty$ be a frame for an infinite dimensional real or complex Hilbert space  $\HH$.  The following
		are equivalent:
		\begin{enumerate}
			\item If $T$ is a trace class self-adjoint operator such that
			$$\langle Tx_k, x_k\rangle=0, \mbox{ for all } k,$$ then $T=0$.
			\item For every $\lambda =(\lambda_1, \lambda_2, \ldots)\in \ell_1$ and for every orthonormal basis 
			$\{e_j\}_{j=1}^\infty$ for $\HH$, if $\sum_{j=1}^\infty\lambda_j|\langle x_k,
			e_j\rangle|^2 =0$ for all $k$ then $\lambda=0$.
		\end{enumerate}
	\end{theorem}
	\begin{proof}
		$(1)\Rightarrow (2)$:
		We prove the result by way of contradiction.  Suppose that (2) is false. Then there is an $\lambda=(\lambda_1, \lambda_2, \ldots)\in \ell_1$ and an
		orthonormal basis $\{e_j\}_{j=1}^\infty$ so that $\sum_{j=1}^\infty\lambda_j|\langle x_k,
		e_j\rangle|^2 =0$ for all $k$ but $\lambda\not=0$.
		
		Define an operator on $\HH$ by
		\[ Te_j = \lambda_je_j, \mbox{ for all }j=1, 2, \ldots.\]
		Then $T$ is a non-zero self-adjoint operator. Moreover, 
		$$|T|e_j=\sqrt{TT^*}e_j=|\lambda_j|e_j, \mbox{ for all }j.$$ Therefore, 
		$$\sum_{j=1}^{\infty}\langle |T|e_j, e_j\rangle=\sum_{j=1}^{\infty}|\lambda_j|<\infty.$$ Thus, $T$ is a non-zero trace class self-adjoint operator. Moreover, we have that  
		$\langle Tx_k,x_k\rangle =\sum_{j=1}^\infty\lambda_j|\langle x_k,
		e_j\rangle|^2 =0$, for all $k=1, 2,\ldots$. This is a contradiction.
		
		$(2)\Rightarrow (1)$:
		Let $T$ be a trace class self-adjoint operator such that  $$\langle Tx_k,x_k\rangle =0, \mbox{ for all }k.$$
		Since $T$ is trace class and hence compact, there
		is an eigenbasis $ \{e_j\}_{j=1}^\infty$ for $T$
		with respective eigenvalues $\{\lambda_j\}_{j=1}^\infty$.  Then
		for every $k=1,2,\ldots$, we have
		\[ \sum_{j=1}^\infty\lambda_j|\langle x_k,
		e_j\rangle|^2=\langle Tx_k,x_k\rangle = 0, \mbox{ for all }k.\]
		
		Since $T$ is trace class then 
		$$\sum_{j=1}^{\infty}|\lambda_j|=\sum_{j=1}^{\infty}|\langle Te_j, e_j\rangle|<\infty.$$
		That is, $\lambda=(\lambda_1, \lambda_2, \ldots)\in \ell_1$. By assumption (2) we get $\lambda=0$ and hence $T=0$.	
	\end{proof}
	
	Finally, by normalizing the trace, we can give a classification for the Injectivity problem if we require further that our operators are trace one. First, we need to justisfy Theorem \ref{TH} so that we can use it for this case.
	
	\begin{theorem}
		Given a family of vectors $\mathcal{X}=\{x_k\}_{k=1}^\infty$ in the real or complex Hilbert space $\HH$, the following are equivalent:
		\begin{enumerate}
			\item Whenever $T, S$ are trace class positive and self-adjoint of trace one, and
			\[ \langle Tx_k,x_k\rangle = \langle Sx_k,x_k\rangle,
			\mbox{ for all k},\]
			then $T=S$.
			\item Whenever $T, S$ are trace class self-adjoint of trace one, and
			\[ \langle Tx_k,x_k\rangle = \langle Sx_k,x_k\rangle,
			\mbox{ for all k},\]
			then $T=S$.
			\item Whenever $T$ is trace class self-adjoint of trace zero, and
			\[ \langle Tx_k,x_k\rangle = 0,
			\mbox{ for all }k,\]
			then $T=0$.
		\end{enumerate}
	\end{theorem}
	\begin{proof}
		$(1)\Rightarrow (2)$:
		Let $T, S$ be trace class self-adjoint operators of trace one such that 	\[ \langle Tx_k,x_k\rangle = \langle Sx_k,x_k\rangle,
		\mbox{ for all }k.\] Set $R=T-S$ then $R$ is a trace class self-adjoint operator of trace zero. Let $\{e_j\}_{j=1}^\infty$ be an orthonormal basis for $\HH$ and let $\{u_j\}_{j=1}^\infty$ be an eigenbasis for $R$
		with respective eigenvalues $\{\lambda_j\}_{j=1}^\infty$. Then $\sum_{j=1}^{\infty}\lambda_j=0$. Define operators $U$ and $D$ on $\HH$ by $Ue_j=u_j$ and $De_j=\lambda_je_j$, for $j=1, 2,\ldots$. Then $U$ is an unitary operator and $D$ is a trace class self-adjoint operator of trace zero, and
		$$R=UDU^*.$$ Now define non-negative numbers 
		
		$$r_1=\dfrac{1+|\lambda_1|}{A}, s_1=\dfrac{1+|\lambda_1|-\lambda_1}{A}, r_j=\dfrac{|\lambda_j|}{A}, s_j=\dfrac{|\lambda_j|-\lambda_j}{A}, j= 2,3, \ldots,$$
		where
		$$A=1+\sum_{j=1}^{\infty}|\lambda_j|=1+\sum_{j=1}^{\infty}|\lambda_j|-\sum_{j=1}^{\infty}\lambda_j>0$$ then $\lambda_j=r_j-s_j$ for all $j$. Let $D_1, D_2$ be operators defined by 
		$$D_1e_j=r_je_j, \ D_2e_j=s_je_j \mbox{  for } j=1, 2, \ldots.$$
		Then
		$D_1, D_2$ are trace class positive self-adjoint of trace one and we have
		$$R=UDU^*=U(D_1-D_2)U^*=UD_1U^*-UD_2U^*.$$
		Moreover, $UD_1U^*,UD_2U^*$ are trace class positive self-adjoint operators of trace one.
		Since 
		$$0=\langle Rx_k, x_k\rangle=\langle UD_1U^*x_k, x_k\rangle-\langle UD_2U^*x_k, x_k\rangle,$$ then $UD_1U^*=UD_2U^*$. Thus, $R=0$ and hence $T=S$.

		\vskip 10pt
		$(2)\Rightarrow (3)$: Let $T$ be any trace class operator of trace zero such that
		$$\langle Tx_k, k_k\rangle=0 \mbox{ for all }k.$$
		Define an operator $S$ on $\HH$ by 
		$$Se_1=e_1, Se_j=0, \mbox{ for } j=2, 3,\ldots.$$
		Then $S$ and $T+S$ are trace class self-adjoint operators of trace one.
		
		Since $\langle (T+S)x_k, x_k\rangle=\langle Sx_k, x_k\rangle$ for all $k$,  $T+S=S$ and hence $T=0$.
		
		\vskip 10pt
		$(3)\Rightarrow (1)$: Let $T, S$ are trace class positive self-adjoint operators of trace one 
		such that
		$$\langle Tx_k,x_k\rangle =\langle Sx_k, x_k\rangle \mbox{ for all  }k.$$
		Then $\langle (T-S)x_k, x_k\rangle=0$ for all $k$. Since $T-S$ is a trace class seft-adjoint operator of trace zero, $T=S$ by $(3)$.
	\end{proof}
	
	Now we are ready to give a classification for the Injectivity problem for operators of trace one. First, we need a definition.
	\begin{definition}
		We define a subspace of the real space $\ell_1$ as follows:
		\[W:=\{(\lambda_1, \lambda_2, \ldots)\in \ell_1: \sum_{j=1}^{\infty}\lambda_j=0\}.\]
	\end{definition}

	\begin{theorem}
		Let $\mathcal{X}=\{x_k\}_{k=1}^\infty$ be a frame for an infinite dimensional real or complex Hilbert space  $\HH$.  The following
		are equivalent:
		\begin{enumerate}
			\item If $T$ is a trace class self-adjoint operator of trace zero such that
			$$\langle Tx_k, x_k\rangle=0, \mbox{ for all } k,$$ then $T=0$.
			\item For every $\lambda =(\lambda_1, \lambda_2, \ldots)\in W$ and for every orthonormal basis 
			$\{e_j\}_{j=1}^\infty$ for $\HH$, if $\sum_{j=1}^\infty\lambda_j|\langle x_k,
			e_j\rangle|^2 =0$ for all $k$ then $\lambda=0$.
		\end{enumerate}
	\end{theorem}
	\begin{proof}
		$(1)\Rightarrow (2)$:
		We prove the contrapositive.  Suppose that (2) is false. Then there is an $\lambda=(\lambda_1, \lambda_2, \ldots)\in W$ and an
		orthonormal basis $\{e_j\}_{j=1}^\infty$ so that $\sum_{j=1}^\infty\lambda_j|\langle x_k,
		e_j\rangle|^2 =0$ for all $k$ but $\lambda\not=0$.
		
		Define an operator on $\HH$ by
		\[ Te_j = \lambda_je_j, \mbox{ for all }j=1, 2, \ldots.\]
		Then $T$ is a non-zero trace class self-adjoint operator of trace zero. Moreover, we have that  
		$\langle Tx_k,x_k\rangle =\sum_{j=1}^\infty\lambda_j|\langle x_k,
		e_j\rangle|^2 =0$, for all $k=1, 2,\ldots$. This is a contradiction.
		
		\vskip 10pt
		$(2)\Rightarrow (1)$:
		Let $T$ be a trace class self-adjoint operator of trace zero such that  $$\langle Tx_k,x_k\rangle =0, \mbox{ for all }k.$$
		Let $ \{e_j\}_{j=1}^\infty$ be an eigenbasis for $T$
		with respective eigenvalues $\{\lambda_j\}_{j=1}^\infty$.  Then
		for every $k=1,2,\ldots$, we have
		\[ \sum_{j=1}^\infty\lambda_j|\langle x_k,
		e_j\rangle|^2=\langle Tx_k,x_k\rangle = 0, \mbox{ for all }k.\]
		
		Since $T$ is trace class, 
		$$\sum_{j=1}^{\infty}|\lambda_j|=\sum_{j=1}^{\infty}|\langle Te_j, e_j\rangle|<\infty.$$
		Moreover, $\sum_{j=1}^{\infty}\lambda_j=0$. Thus, $\lambda=(\lambda_1, \lambda_2, \ldots)\in W$. By assumption (2) we get $\lambda=0$ and hence $T=0$.
	\end{proof}

From now on, 
we will also work in the direct sum of infinitely many copies of
$\ell_2$.

\begin{definition}
	Denote by $\tilde{\HH}$ the direct sum of the real Hilbert spaces $\ell_2$:
	\[ \tilde{\HH}=\left ( \sum_{i=1}^{\infty}\oplus \ell_2 \right )_{\ell_2}.\]
	To avoid confusion with earlier notation, a vector in this direct
	sum will be written in the form:
	\[ \vec{x}=(\vec{x}_1, \vec{x}_2, \ldots, \vec{x}_n, \ldots )	,\]
	and we have
	\[ \langle \vec{x},\vec{y}\rangle = \sum_{i=1}^{\infty}
	\langle \vec{x}_i,\vec{y}_i\rangle.\]
\end{definition}

We also need the following lemma for both the real and complex cases.
\begin{lemma}\label{lem100}
	Let $A=(a_{ij})_{i,j=1}^\infty$ be a real or complex infinite matrix such that $\sum_{i, j=1}^{\infty}|a_{ij}|^2< \infty$. Then the operator $T_A$ defined in $\ell_2$ by 
	$$T_A(x_1,x_2, \ldots)=(y_1,y_2,\ldots),$$ where
	$$y_i=\sum_{j=1}^{\infty}a_{ij}x_j, i=1, 2, \ldots,$$ is a bounded operator. Moreover, $T_A$ is self-adjoint if and only if $a_{ji}=\bar{a}_{ij}$ for all $i, j$.
\end{lemma}
\begin{proof}
	Let $x=\{x_i\}_{i=1}^{\infty}\in \ell_2$. For each $i=1, 2\ldots$, we have
	\[|y_i|^2\leq\left(\sum_{j=1}^{\infty}|a_{ij}x_j|\right)^2\leq\left(\sum_{j=1}^{\infty}|a_{ij}|^2\right)\left(\sum_{j=1}^{\infty}|x_j|^2\right)=\left(\sum_{j=1}^{\infty}|a_{ij}|^2\right)\Vert x\Vert^2.\] 
	Hence, 
	\[\Vert T_Ax\Vert^2= \sum_{i=1}^{\infty}|y_i|^2\leq\left(\sum_{i=1}^{\infty}\sum_{j=1}^{\infty}|a_{ij}|^2\right)\Vert x\Vert^2.\]
	This shows that $T_A$ is a bounded operator on $\ell_2$.
	
 Suppose that $T$ is self-adjoint. Then 
	\[ a_{ji}=\langle T_Ae_i, e_j\rangle=\langle e_i, T_Ae_j\rangle=\overline{\langle T_Ae_j, e_i\rangle}=\bar{a}_{ij},\] for all $i, j$.
	
	Conversely, if $a_{ji}=\bar{a}_{ij}$ for all $i, j$, then
	$$\langle T_A^*e_i, e_j\rangle=\langle e_i, T_Ae_j\rangle=\overline{\langle T_Ae_j, e_i\rangle}=\bar{a}_{ij}=a_{ji}=\langle T_Ae_i, e_j\rangle,$$ for all $i, j$. Hence $T_A^*=T_A$.
\end{proof}

\subsubsection{The real case}
Now we will solve the infinite dimensional injectivity problem
in the real case.  To avoid confusion between coordinates of a
vector in $\ell_2$ and vectors in $\tilde{\HH}$ we define:

\begin{definition}\label{defn real infi}
	For $x=\{x_i\}_{i=1}^\infty\in \ell_2$, we define
	$$\tilde{x}=(\vec{x}_1, \vec{x}_2, \ldots, \vec{x}_n, \ldots )
	\in \tilde{\HH},$$
	where
	$$\vec{x}_1=(x_1x_1, x_1x_2, \ldots); \ \vec{x}_2=(x_2x_2, x_2x_3, \ldots); \ldots;\vec{x}_n=(x_nx_n, x_nx_{n+1}, \ldots); \ldots $$
\end{definition}

We first observe that these vectors are actually in $\tilde{\HH}$.
\begin{lemma}
	If $x=\{x_i\}_{i=1}^\infty\in \ell_2$, then $\tilde{x}\in \tilde{\HH}$.
\end{lemma}
\begin{proof}
	We have that $$\sum_{j=i}^\infty|x_ix_j|^2=|x_i|^2\sum_{j=i}^\infty|x_j|^2\leq |x_i|^2\Vert x\Vert^2,$$ for $i=1, 2, \ldots.$ Hence $\vec{x}_i\in \ell_2$ for all $i$.
	
	Moreover, since
	$$\sum_{i=1}^{\infty}\Vert\vec{x}_i\Vert^2\leq \sum_{i=1}^{\infty}|x_i|^2\Vert x\Vert^2=\Vert x\Vert^4,$$ then $\tilde{x}\in \tilde{\HH}$.
\end{proof}

Now we are ready for the classification of the solutions to the
injectivity problem in the infinite dimensional case.

\begin{theorem}\label{Real inf}
	Let $\mathcal{X}=\{x_k\}_{k=1}^{\infty}$ be a frame in the real Hilbert space $\ell_2$.  The
	following are equivalent:
	\begin{enumerate}
		\item $\mathcal{X}$ is injective.
		\item $\overline{\spn}\{\tilde{x}_k\}_{k=1}^\infty=\tilde{\HH}$.
	\end{enumerate}
		
\end{theorem}

\begin{proof}
	$(1) \Rightarrow (2)$: Let any $a=(\vec{a}_1, \vec{a}_2, \ldots )\in \tilde{\HH}$ be such that $a\perp \overline{\spn}\{\tilde{x}_k\}_{k=1}^\infty$. Then  $\langle a, \tilde{x}_k\rangle=0$ for all $k$.	
	
	We denote
	$$\vec{a}_1=(a_{11}, a_{12}, \ldots); \ \vec{a}_2=(a_{22}, a_{23}, \ldots); \ldots,\vec{a}_n=(a_{nn}, a_{n(n+1)}, \ldots); \ldots.$$ 
	
	Define an infinite matrix $B=(b_{ij})_{i,j=1}^\infty$, where $b_{ii}=a_{ii}$ for all $i$ and $b_{ij}=b_{ji}=\dfrac{1}{2}a_{ij}$ for all $i<j$.
	
	Then by Lemma \ref{lem100}, the operator $T_B$ defined by $B$ is a Hilbert
	Schmidt self-adjoint operator.
	
	For any $x= \{x_i\}_{i=1}^{\infty}\in \ell_2$, we have
	\begin{align*}
	\langle T_Bx,x\rangle&= \sum_{i=1}^{\infty}\sum_{j=1}^{\infty}b_{ij}x_ix_j\\
	&=\sum_{i=1}^{\infty}b_{ii}x_i^2+2\sum_{i<j}b_{ij}x_ix_j\\
	&=\sum_{i=1}^{\infty}a_{ii}x_i^2+\sum_{i<j}a_{ij}x_ix_j\\
	&=\sum_{i=1}^{\infty}\langle \vec{a}_i, \vec{x}_i\rangle\\
	&=\langle a, \tilde{x}\rangle.
	\end{align*}
	Hence, 
	$\langle T_Bx_k, x_k\rangle=\langle a, \tilde{x}_k\rangle=0$ for all $k$. This implies $T_B=0$ by (1) and therefore $a=0$.
	
	$(2)\Rightarrow (1)$: Let $T$ be a Hilbert Schmidt
 self-adjoint operator on $\ell_2$ such that $\langle Tx_k, x_k\rangle=0$ for all $k$, and recall that $\{e_i\}_{i=1}^\infty$ is the canonical orthonormal basis for $\ell_2$. 
	
	Denote 
	$$a_{ij}=\langle Te_j, e_i\rangle, i, j=1, 2\ldots,$$ and 
	$$\tilde{T}=(\vec{a}_1, \vec{a}_2, \ldots, \vec{a}_n, \ldots),$$
	where
	$$\vec{a}_1=(a_{11}, 2a_{12}, 2a_{13},\ldots); \quad \vec{a}_2=(a_{22}, 2a_{23},2a_{24}, \ldots);\ldots; $$
	$$\vec{a}_n =(a_{nn}, 2a_{n(n+1)}, 2a_{n(n+2)}, \ldots); \ldots.$$
	Since $T$ is a Hilbert Schmidt operator, $\tilde{T}\in \tilde{\HH}$. 
	Moreover, we have 
	$$\langle \tilde{T}, \tilde{x}_k\rangle=\langle Tx_k, x_k\rangle=0, \mbox{ for all } k.$$ Since $\overline{\spn}\{\tilde{x}_k\}_{k=1}^\infty=\tilde{\HH}$, we get $\tilde{T}=0$. So $T=0$. 
	
\end{proof}

\begin{remark}
We have that $(2)\Rightarrow (1)$ in the theorem holds for
trace class operators.  But in general $(1)\Rightarrow (2)$ since
the operators we construct may not be trace class.
\end{remark}

\subsubsection{The complex case}

For the complex case of the injectivity problem, we need a new
variation of the tilde vectors.

\begin{definition}\label{defn complex infi}
	For $x=\{x_i\}_{i=1}^\infty\in \ell_2$, we define
	$$\tilde{x}=(\vec{x}_1, \vec{x}_2, \ldots, \vec{x}_n, \ldots ),$$
	where
	\[\vec{x}_1= (|x_1|^2,\R(\bar{x}_1x_2), \I(\bar{x}_1x_2),\R(\bar{x}_1x_3),\I(\bar{x}_1x_3),\ldots);\]
	\[\vec{x}_2= (|x_2|^2,\R(\bar{x}_2x_3), \I(\bar{x}_2x_3),\R(\bar{x}_2x_4),\I(\bar{x}_2x_4),\ldots);\ldots;\]
	\[\vec{x}_n= (|x_n|^2,\R(\bar{x}_nx_{n+1}), \I(\bar{x}_nx_{n+1}),\R(\bar{x}_nx_{n+2}),\I(\bar{x}_nx_{n+2}),\ldots);\ldots.\]
	
\end{definition}
We first need to verify that our vectors are in $\tilde{\HH}$.

\begin{lemma}
	If $x=\{x_i\}_{i=1}^\infty\in \ell_2$, then $\tilde{x}\in \tilde{\HH}$.
\end{lemma}
\begin{proof}
	For each $i=1, 2, \ldots,$ we have \begin{align*}
	\Vert\vec{x}_i\Vert^2&=|x_i|^4+\sum_{j=i+1}^\infty|\R(\bar{x}_ix_j)|^2+\sum_{j=i+1}^\infty|\I(\bar{x}_ix_j)|^2\\
	&=|x_i|^4+\sum_{j=i+1}^\infty|\bar{x}_ix_j|^2\\
	&=|x_i|^2\left(|x_i|^2+\sum_{j=i+1}^\infty|x_j|^2\right)\\
	&\leq |x_i|^2\Vert x\Vert^2.\end{align*} 
	It follows that:
	\[ \sum_{i=1}^{\infty}\|\vec{x}_i\|^2 \le  \sum_{i=1}^{\infty}
	|x_i|^2\|x\|^2= \|x\|^4.\]
	This implies $\tilde{x}\in \tilde{\HH}$.	
\end{proof}

Now we give the classification theorem for injectivity in the
infinite dimensional case.

\begin{theorem}\label{complex infinite}
	Let $\mathcal{X}=\{x_k\}_{k=1}^{\infty}$ be a frame in the complex Hilbert space $\ell_2$.  The
	following are equivalent:
	\begin{enumerate}
		\item $\mathcal{X}$ gives injectivity.
		\item $\overline{\spn}\{\tilde{x}_k\}_{k=1}^\infty=\tilde{\HH}$.
	\end{enumerate}
\end{theorem}
\begin{proof}
	$(1) \Rightarrow (2)$: Let any $a=(\vec{a}_1, \vec{a}_2, \ldots )\in \tilde{\HH}$ be such that $a\perp \overline{\spn}\{\tilde{x}_k\}_{k=1}^\infty$. Then  $\langle a, \tilde{x}_k\rangle=0$ for all $k$.	
	
	Denote
	\[\vec{a}_1=(a_{11}, u_{12}, v_{12}, u_{13}, v_{13}, \ldots);\
	\vec{a}_2=(a_{22}, u_{23}, v_{23}, u_{24}, v_{24}, \ldots);\ldots;\]
	\[\vec{a}_n=(a_{nn}, u_{n(n+1)}, v_{n(n+1)}, u_{n(n+2)}, v_{n(n+2)}, \ldots);\ldots.\]
	
	Define an infinite matrix $B=(b_{ij})_{i,j=1}^\infty$, where $b_{ii}=a_{ii}$ for all $i$ and $b_{ij}=\bar{b}_{ji}=\dfrac{1}{2}(u_{ij}-\iota v_{ij})$ for all $i<j$.
	
	We have
	\begin{align*}
	\sum_{i,j=1}^{\infty}|b_{ij}|^2=\sum_{i=1}^{\infty}|a_{ii}|^2+2\sum_{i<j}|b_{ij}|^2=\sum_{i=1}^{\infty}|a_{ii}|^2+\dfrac{1}{2}\sum_{i<j}\left(|u_{ij}|^2+|v_{ij}|^2\right)<\infty.
	\end{align*}
	
	Then by Lemma \ref{lem100}, the operator $T_B$ defined by $B$ is Hilbert
	Schmidt and self-adjoint.

	For any $x=\{x_i\}_{i=1}^\infty\in \ell_2$, we have
	\begin{align*}
	\langle T_Bx,x\rangle &= \sum_{i=1}^\infty \sum_{j=1}^\infty b_{ij}\bar{x}_ix_j\\
	&=\sum_{i=1}^\infty b_{ii}|x_i|^2 + 2\sum_{i<j}\R( b_{ij}\bar{x}_ix_j)\\
	&=\sum_{i=1}^\infty b_{ii}|x_i|^2 + 2\sum_{i<j}(\R( b_{ij})\R(\bar{x}_ix_j)-\I( b_{ij})\I(\bar{x}_ix_j))\\
	&=\sum_{i=1}^{\infty}a_{ii}|x_i|^2+\sum_{i<j }\left(u_{ij}\R(\bar{x}_ix_j)+v_{ij}\I(\bar{x}_ix_j)\right)\\
	&=\sum_{i=1}^{\infty}\langle \vec{a}_i, \vec{x}_i\rangle\\
	&=\langle a,\tilde{x}\rangle.
	\end{align*}
	Hence, 
	$\langle T_Bx_k, x_k\rangle=\langle a, \tilde{x}_k\rangle=0$ for all $k$. This implies $T_B=0$ by (1) and therefore $a=0$.
	
	$(2)\Rightarrow (1)$: Let $T$ be a Hilbert Schmidt self-adjoint operator such that $\langle Tx_k, x_k\rangle=0$ for all $k$.
	
	Denote 
	$$a_{ij}=\langle Te_j, e_i\rangle, i, j=1, 2,\ldots,$$ and 
	$$\tilde{T}=(\vec{a}_1, \vec{a}_2, \ldots, \vec{a}_n, \ldots),$$
	where
	$$\vec{a}_1=(a_{11}, 2\R(a_{12}), -2\I(a_{12}), 2\R(a_{13}), -2\I(a_{13}),\ldots); $$
	$$\vec{a}_2=(a_{22}, 2\R(a_{23}), -2\I(a_{23}),2\R(a_{24}), -2\I(a_{24}), \ldots);\ldots; $$
	$$\vec{a}_n =(a_{nn}, 2\R(a_{n(n+1)}), -2\I(a_{n(n+1)}),2\R(a_{n(n+2)}), -2\I(a_{n(n+2)}), \ldots); \ldots$$
	Since $T$ is Hilbert Schmidt, $\tilde{T}\in \tilde{\HH}$. 
	
	For any $x=\sum_{j=1}^{\infty}x_je_j$ we have
	\begin{align*}\langle Tx, x\rangle&=\sum_{i=1}^{\infty}\sum_{j=1}^\infty\bar{x}_ix_j\langle Te_j, e_i\rangle\\
	&=\sum_{i=1}^{\infty}\sum_{j=1}^\infty\bar{x}_ix_ja_{ij}\\
	&=\sum_{i=1}^\infty a_{ii}|x_i|^2 + 2\sum_{i<j}\R( a_{ij}\bar{x}_ix_j)\\
	&=\sum_{i=1}^\infty a_{ii}|x_i|^2 + 2\sum_{i<j}(\R( a_{ij})\R(\bar{x}_ix_j)-\I( a_{ij})\I(\bar{x}_ix_j))\\
	&=\langle \tilde{T}, \tilde{x}\rangle.
	\end{align*}
	Hence
	$$\langle \tilde{T}, \tilde{x}_k\rangle=\langle Tx_k, x_k\rangle=0, \mbox{ for all } k.$$ Since $\overline{\spn}\{\tilde{x}_k\}_{k=1}^\infty=\tilde{\HH}$, $\tilde{T}=0$. So $T=0$. This completes the proof.
\end{proof}

As a consequence we have:

\begin{corollary}
For a frame $\{x_k\}_{k=1}^{\infty}$ in $\ell_2$ the following
are equivalent:
\begin{enumerate}
\item The family $\{x_kx_k^*\}_{k=1}^{\infty}$ spans the family
of real self-adjoint Hilbert Schmidt operators on $\ell_2$.
\item The family $\{\tilde{x}_k\}_{k=1}^{\infty}$ spans $\tilde{\HH}$. 
\end{enumerate}
\end{corollary}

\begin{remark}
		As we have seen in the proof of Theorem \ref{Real inf} for the real case and Theorem \ref{complex infinite} for the complex case, for a vector $a\in \tilde{\HH}$, there is a Hilbert Schmidt self-adjoint operator $T$ so that
		$$\langle Tx, x\rangle=\langle a, \tilde{x}\rangle, \mbox{ for all } x\in \ell_2.$$
		
		Conversely, for a Hilbert Schmidt self-adjoint  operator $T$, there is a vector $\tilde{T}\in \tilde{\HH}$ satisfying
		$$\langle \tilde{T}, \tilde{x}\rangle=\langle Tx, x\rangle, \mbox{ for all } x\in \ell_2.$$
	\end{remark}

 Is is easy to see that the canonical orthonormal basis is not injective. Actually, the family $\{\tilde{e}_i\}_{i=1}^\infty$ forms an orthonormal set in $\tilde{\HH}$. We will see in general that any frame in $\ell_2$ so that the corresponding tilde vectors form a frame sequence in $\tilde{\HH}$ cannot be injective.

	\begin{theorem}\label{thm103}
		For any Bessel sequence $\{x_k\}_{k=1}^{\infty}$ for the real or complex space $\ell_2$, the family
		$\{\tilde{x}_k\}_{k=1}^{\infty}$ is a Bessel sequence in $\tilde{\HH}$. However, $\{\tilde{x}_k\}_{k=1}^{\infty}$ is not a frame for $\tilde{\HH}$.
	\end{theorem}
	\begin{proof} We may assume that $\|x_k\|\le 1$ for all $k$.
	Let $B$ be the Bessel bound of $\{x_k\}_{k=1}^\infty$.
		
		Given any finite real scalar sequence 
		$\{a_k\}_{k=1}^{\infty}$ we will compute the real case and the complex case separately.
		
		\vskip10pt
		{\bf The real case:} Using Definition \ref{defn real infi} for the tilde vector $\tilde{x}_k$, we have 
		\begin{align*}
		\|\sum_{k=1}^{\infty}a_k\tilde{x}_k\|^2 &=
		\sum_{i=1}^{\infty}\sum_{j=i}^{\infty}\left (\sum_{k=1}^{\infty}
		a_kx_{ki}x_{kj}\right )^2\\
		&\le \sum_{i=1}^{\infty}\sum_{j=1}^{\infty}\left ( \sum_{k=1}^{\infty}
		a_kx_{ki}x_{kj}\right )^2\\
		&=\sum_{i=1}^{\infty}\|\sum_{k=1}^{\infty}a_kx_{ki}x_k\|^2.	
		\end{align*}
		Using the fact that $\{x_k\}_{k=1}^\infty$ is Bessel with bound $B$, we get
		\begin{align*}
		\|\sum_{k=1}^{\infty}a_k\tilde{x}_k\|^2 &\le B \sum_{i=1}^{\infty}\sum_{k=1}^{\infty}(a_kx_{ki})^2\\
		&=B\sum_{k=1}^{\infty}a_k^2\sum_{i=1}^{\infty}x_{ki}^2\\
		&= B \sum_{k=1}^{\infty}a_k^2 \|x_k\|^2\\
		&\le B \sum_{k=1}^{\infty}a_k^2.
		\end{align*}
		
		{\bf The complex case:} Now we need to use Definition \ref{defn complex infi} for the tilde vectors $\tilde{x}_k$. We have that
		\begin{align*}
		\|\sum_{k=1}^{\infty}a_k\tilde{x}_k\|^2 &=
		\sum_{i=1}^{\infty}\left ( \sum_{k=1}^{\infty}
		a_k|x_{ki}|^2\right )^2+\sum_{i=1}^{\infty}\sum_{j=i+1}^{\infty}\left ( \sum_{k=1}^{\infty}
		a_k\R(\bar{x}_{ki}x_{kj})\right )^2\\
		&+\sum_{i=1}^{\infty}\sum_{j=i+1}^{\infty}\left ( \sum_{k=1}^{\infty}
		a_k\I(\bar{x}_{ki}x_{kj})\right )^2\\
		&=\sum_{i=1}^{\infty}\left ( \sum_{k=1}^{\infty}
		a_k|x_{ki}|^2\right )^2+\sum_{i=1}^{\infty}\sum_{j=i+1}^{\infty}\left(\R\left ( \sum_{k=1}^{\infty}
		a_k\bar{x}_{ki}x_{kj}\right)\right )^2\\
		&+\sum_{i=1}^{\infty}\sum_{j=i+1}^{\infty}\left(\I\left ( \sum_{k=1}^{\infty}
		a_k\bar{x}_{ki}x_{kj}\right)\right )^2\\
		&\le 2\sum_{i=1}^{\infty}\left ( \sum_{k=1}^{\infty}
		a_k|x_{ki}|^2\right )^2+2\sum_{i=1}^{\infty}\sum_{j=i+1}^{\infty}\left | \sum_{k=1}^{\infty}
		a_k\bar{x}_{ki}x_{kj}\right|^2\\
		&\le 2\sum_{i=1}^{\infty}\left ( \sum_{k=1}^{\infty}
		a_k|x_{ki}|^2\right )^2+2\sum_{i=1}^{\infty}\sum_{j=1, j\not= i}^{\infty}\left | \sum_{k=1}^{\infty}
		a_k\bar{x}_{ki}x_{kj}\right|^2\\
		&=2\sum_{i=1}^{\infty}\|\sum_{k=1}^{\infty}a_k\bar{x}_{ki}x_k\|^2
		\le 2B \sum_{i=1}^{\infty}\sum_{k=1}^{\infty}a_k^2|x_{ki}|^2\\
		&=2B\sum_{k=1}^{\infty}a_k^2\sum_{i=1}^{\infty}|x_{ki}|^2
		= 2B \sum_{k=1}^{\infty}a_k^2 \|x_k\|^2
		\le 2B \sum_{k=1}^{\infty}a_k^2.
		\end{align*}
		
		Hence, $\{\tilde{x}_k\}_{k=1}^{\infty}$ is a Bessel sequence for the both cases. 
		
		\vskip10pt
		Now we will show that $\{\tilde{x}_k\}_{k=1}^{\infty}$ fails to have a lower frame bound. We will prove the real case, the complex case is similar.
		By our assumption, we have that
		\[ \sum_{i=1}^{\infty}|x_{ki}|^2<\infty, \mbox{ for all }k.\]
		Also,
		\[ \sum_{k=1}^{\infty}|x_{ki}|^2 = \sum_{k=1}^{\infty}|\langle
		e_i,x_k\rangle|^2 < \infty, \mbox{ for all }i.\]
		Fix $\epsilon>0$ and choose $n$ so that
		\[ \sum_{k=n}^{\infty}|x_{k1}|^2 < \epsilon.\]
		Now choose $m$ so that
		\[ \sum_{k=1}^{n-1}|x_{km}|^2< \epsilon.\]
		Let $$\tilde{e}_{1m}:=(e_m; \vec{0};\vec{0}, \ldots)\in \tilde{\HH}.$$ 
		Then we have
		\begin{align*}
		\sum_{k=1}^{\infty}|\langle \tilde{e}_{1m},\tilde{x}_k\rangle|^2
		&= \sum_{k=1}^{\infty}|x_{k1}|^2|x_{km}|^2\\
		&= \sum_{k=1}^{n-1}|x_{k1}|^2|x_{km}|^2 + \sum_{k=n}^{\infty}|x_{k1}|^2|x_{km}|^2\\
		&\le \sum_{k=1}^{n-1}|x_{km}|^2+\sum_{k=n}^{\infty}|x_{k1}|^2\\
		&< 2 \epsilon.
		\end{align*}
		It follows that $\{\tilde{x}_k\}_{k=1}^{\infty}$ does not have a
		lower frame bound. 
	\end{proof}
	\begin{corollary}\label{Coro10}
		Let $\{x_k\}_{k=1}^\infty$ be a frame for $\ell_2$. If $\{\tilde{x}_k\}_{k=1}^\infty$ is a frame sequence, then $\{x_k\}_{k=1}^\infty$ cannot be injective.
	\end{corollary}
	\begin{proof} We will prove the real case, the complex case is similar.
		
		Suppose by way of contradiction that $\{x_k\}_{k=1}^\infty$ is injective. Then by Theorem \ref{Real inf}, $\{\tilde{x}_k\}_{k=1}^\infty$ spans $\tilde{\HH}$. Thus, it is a frame for $\tilde{\HH}$, which contradicts Theorem \ref{thm103}.
	\end{proof}

\subsection{Constructing the Solutions to the Injectivity Problem}

For the construction of solutions to the injectivity problem,
we will follow the outline for the finite dimensional case.
But this construction is much more complicated because of
problems with convergence, problems with keeping the upper frame
bound finite, and the fact that we cannot show spanning in
$\ell_2$ by just checking linear independence.  Also, we proved 
in the finite dimensional case that the $\tilde{x}_i$ span
by showing they are independent and have enough vectors to
span $\tilde{\HH}$.  This does not work in the infinite dimensional case.  Note that the
following construction works for trace class operators and for
Hilbert Schmidt operators.

 \begin{theorem}\label{inf ex}
	Let $\{e_i\}_{i=1}^\infty$ be the canonial basis for the real Hilbert space $\ell_2$ and let $a_i\not=0$ for $i=1, 2, \ldots$ be such that $\sum_{i=1}^{\infty}a^2_i<\infty.$ Define 
	$$x_k=a_k(e_1+e_{k+1}), \mbox{ for } k=1, 2, \ldots.$$ Let $L$ be the right shift operator on $\ell_2$. 
	Then the family 
	\[ \{e_{i}\}_{i=1}^{\infty}\cup \{\frac{1}{2^i}L^ix_k\}_{i=0,k=1}^{\infty, \ \ \infty}\]
	is a frame for $\ell_2$ which gives injectivity.
\end{theorem}
\begin{proof}
	First we need to see that our family of vectors forms a frame for $\ell_2$. Since our family contains an orthonormal basis for
	$\ell_2$, we automatically have a lower frame bound.  So we need
	to check that our family is Bessel, and since $\{e_i\}_{i=1}^{\infty}$ is already Bessel, we only need to check that 
	$\{\frac{1}{2^i}L^ix_k\}_{i=0, k=1}^{ \infty, \ \ \infty}$ is Bessel.

	For any $x\in \ell_2$, we have
	\begin{align*}
	\sum_{i=0}^{\infty}\sum_{k=1}^{\infty}|\langle x,\frac{1}{2^i}L^{i}x_k\rangle|^2&\le\sum_{i=0}^{\infty}\sum_{k=1}^{\infty}\dfrac{1}{4^i}\Vert x\Vert^2\Vert L^{i}x_k\Vert^2\\
	&\le\sum_{i=0}^{\infty}\sum_{k=1}^{\infty}\dfrac{1}{4^i}\Vert x\Vert^24a_k^2\\
	&=\left(\sum_{i=0}^{\infty}\dfrac{1}{4^{i-1}}\sum_{k=1}^{\infty}a_k^2\right)\Vert x\Vert^2.
	\end{align*}
	
	So our family is a Bessel sequence.
	
	To see our frame is injective, let $T$ be a Hilbert Schmidt self-adjoint operator such that 
	$$\langle Te_k, e_k\rangle=0  \mbox{ and } \langle T(L^ix_k), L^ix_k\rangle=0, \mbox{ for } i=0,1 \ldots;\ k=1, 2, \ldots.$$
	Note that 
	$$L^ix_k=a_k(e_{1+i}+ e_{1+i+k}) \mbox{ for all }  i, k.$$
	Hence
	\begin{align*}
	\langle T(L^ix_k), L^ix_k\rangle&= a_k^2\langle Te_{1+i},e_{1+i}\rangle+2a_k^2\langle Te_{1+i}, e_{1+i+k} \rangle+a_k^2\langle Te_{1+i+k},e_{1+i+k}\rangle\\
	&=2a_k^2\langle Te_{1+i}, e_{1+i+k} \rangle,
	\end{align*}
	for all $i, k$.
	
	This implies $\langle Te_j, e_k\rangle=0$ for all $j, k=1, 2, \ldots,$ and hence $T=0$.
\end{proof}
The complex version of this construction looks like:

\begin{theorem}
	Let $\{e_i\}_{i=1}^{\infty}$ be the canonical orthonormal basis
	for complex $\ell_2$, and let $\{a_i\}_{i=1}^{\infty}, \{b_i\}_{i=1}^\infty\in \ell_2$,   $|a_i|, |b_i|\not=0$ for all $i$. Then the following frame gives
	injectivity:
	\[ \{e_i\}_{i=1}^{\infty}\cup \{\dfrac{1}{2^i}L^i(a_k(e_1+e_{k+1}))\}_{i=0,k=1}^{\infty}
	\cup \{\dfrac{1}{2^i}L^i(b_k(e_1+\iota e_{k+1}))\}_{i=0,k=1}^{\infty}.\]
	
\end{theorem}

The above frames are unbounded.  
The following theorem shows that we can easily adjust unbounded
injective frames to produce bounded injective frames.

\begin{theorem}
	Every injective frame $\mathcal{X}=\{e_i\}_{i=1}^\infty\cup\{x_k\}_{k=1}^{\infty}$ of finitely
	supported vectors, induces a bounded injective frame.  
\end{theorem}

\begin{proof}
	 Recall that for each $k$, we denote
	$$x_{k}=(x_{k1}, x_{k2},\ldots, x_{ki}, \ldots).$$ Choose integer $n_1 < n_2 < \cdots$
	so that 
	\[ \max\{i:x_{ki}\not= 0\}< n_k.\]
	For $k=1,2,\ldots$ let
	\[ y_{2k}=x_k+e_{n_k}\mbox{ and }y_{2k+1}=x_k-e_{n_k},
	\mbox{ for }k=1,2,\ldots.\]
	It is clear that $\mathcal{Y}=\{e_i\}_{i=1}^\infty\cup\{y_k\}_{k=1}^{\infty}$ is still a frame and
	$\|y_k\|\geq 1$ for all $k=1,2,\ldots$.  Since 
	$\tilde{y}_{2k}+\tilde{y}_{2k+1}=2 \tilde{x}_k+2\tilde{e}_{n_k}$, and $\{\tilde{e}_{n_k}\}_{k=1}^{\infty}$ are vectors in
	our set, it follows that $\{\tilde{x}_k\}_{k=1}^{\infty}$
	is in our set of vectors and so 
	$\mathcal{Y}$ is injective.
\end{proof}

\subsection{The Solutions are Neither Open nor Dense}

In this section we will show that the solutions to the injectivity problem in infinite dimensions are neither open nor
dense in the class of frames.  

First we need a definition:

\begin{definition}
	Given frames $\mathcal{X}=\{x_k\}_{k=1}^{\infty}$ and
	$\mathcal{Y}=\{y_k\}_{k=1}^{\infty}$ for $\ell_2$, we define the {\bf distance
		between them} by
	\[ d^2(\mathcal{X},\mathcal{Y}) = \sum_{k=1}^{\infty}
	\|x_k-y_k\|^2.\]
	Note that this distance may be infinity.
\end{definition}

The following theorem shows that the frames which give injectivity are not open in the
family of frames for $\ell_2$.

\begin{theorem} Let $ \mathcal{X}=\{e_{i}\}_{i=1}^{\infty}\cup \{\frac{1}{2^i}L^ix_k\}_{i=0,k=1}^{\infty}$ be the injective frame for the real space $\ell_2$ as in Theorem \ref{inf ex}. Then for any $\epsilon>0$, there is a frame $\mathcal{Y}$ such that $d(\mathcal{X}, \mathcal{Y})<\epsilon$, and $\mathcal{Y}$ is not injective.
\end{theorem}

\begin{proof}
	Let any $\epsilon>0$. Since the series $\sum_{i=0}^{\infty}\sum_{k=1}^{\infty}\|\frac{1}{2^i}L^ix_k\|^2$ converges, for any $\epsilon$, there exists $n_0$ such that 
	\[ \sum_{i= n_0+1}^\infty\sum_{k=1}^{\infty}\|\frac{1}{2^i}L^ix_k\|^2< \epsilon^2.\]
	Let
	$y_{ik} = \frac{1}{2^i}L^ix_k$
	for $i=0,1,\ldots,n_0$ and $k=1,2,\ldots$, and
	$y_{ik}=0$ otherwise.  It is clear that 
	\[\mathcal{Y}= \{e_i\}_{i=1}^{\infty} \cup \{y_{ik}\}_{i=0, k=1}^{\infty}\]
	cannot give injectivity by Theorem \ref{Real inf} while
	\[ d^2(\mathcal{X}, \mathcal{Y})=\sum_{i= n_0+1}^\infty\sum_{k=1}^{\infty}\|\frac{1}{2^i}L^ix_k\|^2<\epsilon^2.\]
	This completes the proof.
\end{proof}

\begin{remark}
	There is a perturbation theory for frames which looks like it
	should apply here.  The problem is that although our vectors
	form a frame for $\ell_2$, their tilde vectors do not form a
	frame to $\tilde{\HH}$ and so the theory does not apply.
\end{remark}

To show the solutions are not dense, 
we need the definition of a Riesz sequence in $\ell_2$.

\begin{definition}
A family of vectors $\{x_i\}_{i=1}^{\infty}$ in the real or complex Hilbert space $\ell_2$ is called
 a {\bf Riesz sequence} if there are constants $0<A\le B <
\infty$ so that for all $\{a_i\}_{i=1}^{\infty}\subset \ell_2$ we
have:
\[ A \sum_{i=1}^{\infty}|a_i|^2 \le \|\sum_{i=1}^{\infty}a_ix_i\|^2
\le B \sum_{i=1}^{\infty}|a_i|^2.\]
The constants $A,B$ are called the {\bf lower and upper Riesz bounds}.  If the vectors span $\ell_2$, this is called a
{\bf Riesz basis}. 
\end{definition}

\begin{remark}
It is known \cite{C1,Ch} that a Riesz basis is a frame and the Riesz
bounds are the frame bounds.  Also, $\{x_i\}_{i=1}^{\infty}$ is
a Riesz sequence if and only if the operator $T:\ell_2
\rightarrow \ell_2$ given by $Te_i=x_i$ is a bounded, linear,
invertible operator (on its range).
\end{remark}

Also, we need a perturbation result from frame theory.

\begin{proposition}
Assume $\chi=\{x_i\}_{i=1}^{\infty}$ are vectors in the real or complex space $\ell_2$ satisfying:
\[ \sum_{i=1}^{\infty}\|e_i-x_i\|^2 < \epsilon^2.\]
Then $\chi$ is a Riesz sequence in $\ell_2$ with lower Riesz bound
$(1-\epsilon)^2$.
\end{proposition}

\begin{proof}
We compute for scalars $\{a_i\}_{i=1}^{\infty}$,
\begin{align*}
\|\sum_{i=1}^{\infty}a_ix_i\| &\ge \|\sum_{i=1}^{\infty}a_ie_i\|-
\|\sum_{i=1}^{\infty}a_i(e_i-x_i)\|\\
&\ge \left (\sum_{i=1}^{\infty}|a_i|^2\right )^{1/2}-
\sum_{i=1}^{\infty}|a_i|\|e_i-x_i\|\\
&\ge \left (\sum_{i=1}^{\infty}|a_i|^2\right )^{1/2}-
\left ( \sum_{i=1}^{\infty}|a_i|^2 \right )^{1/2}\left (
\sum_{i=1}^{\infty}\|e_i-x_i\|^2 \right )^{1/2}\\
&\ge \left (\sum_{i=1}^{\infty}|a_i|^2\right )^{1/2}(1-\epsilon)
\end{align*}
The upper Riesz bound is done similarly.
\end{proof}

We also need a theorem from \cite{CaK}.

\begin{theorem}\label{thm7}
Let $Y,Z$ be subspaces of a Banach space $X$.  If $T:Y\rightarrow Z$
is a surjective linear operator with $\|I-T\|<1$, then
$codim_XY=codim_XZ$.
\end{theorem}

The next theorem shows that the solution set of the infinite dimensional
injectivity problem is not dense in the class of all
frames for $\ell_2$.

\begin{theorem}
Let $\{e_k\}_{k=1}^{\infty}$ be the canonical basis for the real space $\ell_2$ and $\mathcal{X}=\{x_k\}_{k=1}^{\infty}\subset \ell_2$ be such that 
\[ \sum_{k=1}^{\infty}\|e_k-x_k\|^2 \leq\frac{1}{8},\]
Then $\mathcal{X}$ is not injective.
\end{theorem}

\begin{proof}
Will will show that $\mathcal{X}$ does not satisfy Theorem \ref{Real inf}. Note that $codim_{\tilde{\HH}}\{\tilde{e}_k\}_{k=1}^{\infty}$ is infinite.  Also, $\{\tilde{e}_k\}_{k=1}^{\infty}$ is an orthonormal
sequence in $\tilde{\HH}$.  We have that
\[ \sum_{k=1}^{\infty}\|e_k-x_k\|^2 =
 \sum_{k=1}^{\infty}\left((1-x_{kk})^2 +\sum_{i\not= k}x_{ki}^2\right)
 \leq\frac{1}{8}.\]
 In particular, $\|x_k\|^2 \le 2$.
Let
\[ X=\overline{\spn}\{\tilde{e}_k\}_{k=1}^{\infty},\mbox{ and }
Y=\overline{\spn}\{\tilde{x}_k\}_{k=1}^{\infty}.\]
For each $k=1, 2, \ldots$ we have
\begin{align*}
\|\tilde{e}_k-\tilde{x}_k\|^2&=(1-x^2_{kk})^2+\sum_{j\geq k+1}(x_{kk}x_{kj})^2+\sum_{i\not=k}\sum_{j\geq i}(x_{ki}x_{kj})^2\\
&\leq(1-x_{kk})^2(2\|x_k\|^2+2)+\|x_k\|^2\sum_{j\geq k+1}x^2_{kj}+\|x_k\|^2\sum_{i\not=k}x^2_{ki}\\
&\leq 6\left((1-x_{kk})^2+\sum_{i\not=k}x^2_{ki}\right).
\end{align*}
Hence,
\begin{align*}
 \sum_{k=1}^{\infty}\|\tilde{e}_k-\tilde{x}_k\|^2&\le 6\sum_{k=1}^{\infty}\left((1-x_{kk})^2 +\sum_{i\not= k}x_{ki}^2\right)\leq \dfrac{3}{4}.
 \end{align*} 
 It follows that $\{\tilde{x}_k\}_{k=1}^{\infty}$ is a Riesz sequence. 
 
 Now we
define $T:X \rightarrow Y$ by: for $x=\sum_{k=1}^{\infty}\langle x, \tilde{e}_k\rangle\tilde{e}_k\in X$, 
\[ Tx=\sum_{k=1}^{\infty}\langle x, \tilde{e}_k\rangle\tilde{x}_k.\]
Since $T$ is mapping a Riesz sequence to a Riesz sequence, it follows
that $T$ is bounded and surjective.
Now,
\begin{align*}
\|(I-T)x\| &= \left \|\sum_{k=1}^{\infty}\langle x, \tilde{e}_k\rangle(\tilde{e}_k-\tilde{x}_k)
\right \|\\
&\leq \sum_{k=1}^{\infty}|\langle x, \tilde{e}_k\rangle|\|\tilde{e}_k-\tilde{x}_k\|\\
&\leq \left ( \sum_{k=1}^{\infty}|\langle x, \tilde{e}_k\rangle|^2 \right )^{1/2}
\left ( \sum_{k=1}^{\infty}\|\tilde{e}_k-\tilde{x}_k\|^2 \right )^{1/2}\\
&\leq \dfrac{\sqrt{3}}{2}\|x\|.
\end{align*}
Hence, $\|I-T\|<1$ and by Theorem \ref{thm7}, $codim_{\tilde{\HH}}Y = 
codim_{\tilde{\HH}}X=\infty.$

\end{proof}

\subsection{The Solution to the State Estimation Problem}

For the infinite dimensional case, the state estimation problem asks:
\vskip10pt
\noindent {\bf State Estimation Problem}: Given an injective Parseval frame $\{x_k\}_{k=1}^{\infty}$ for $\ell_2$, and a sequence of real numbers $a=\{a_k\}_{k=1}^\infty$, 
does there exist a 
Hilbert Schmidt self-adjoint operator $T$ so that
\[ \langle Tx_k, x_k\rangle=a_k,\mbox{ for all } k?\]

\begin{remark}  This problem is rarely solvable.
	\begin{enumerate}
		     \item If $x_kx_k^*=x_lx_l^*$, but $a_k\not= a_l$ for some $k, l$, then the problem has no solution.
			\item Recall a set of vectors $\{x_i\}_{i=1}^{\infty}$ is
		{\bf $\omega$-independent} if $\sum_{i=1}^{\infty}c_ix_i=0$
		implies $c_i=0$ for all $i=1,2,\ldots$.  If $\{x_kx_k^*\}_{k=1}^\infty$ is
		not $\omega$-independent and $\sum_{k=1}^{\infty}c_kx_kx_k^*=0$ but
		not all $c_k$ are zero, then for $\langle Tx_k, x_k\rangle=a_k$ we need
		\[ \sum_{k=1}^{\infty}
		c_ka_k=\langle T,\sum_{k=1}^{\infty}c_kx_kx_k^*\rangle= 0.\]\\	
	\end{enumerate}
\end{remark}

For the solution of the state estimation problem we will need the
notion of a separated sequence in $\ell_2$.

\begin{definition}
	
	A family of vectors $\{x_i\}_{i=1}^{\infty}$ in $\ell_2$ is
	{\bf separated} if for every $j\in \NN$,
	\[ x_j \notin \overline{\spn}\{x_i\}_{i\not= j}.\]
	It is $\delta$-separated if the projection $P_j$ onto
	$\overline{\spn}\{x_i\}_{i\not= j}$ satisfies
	\[ \|(I-P_j)x_j\|\ge \delta.\]
\end{definition}

	

\begin{remark}
In general, a Bessel sequence which is $\delta$-separated may
not be a Riesz sequence.  To see this let
\[ \HH=\left ( \sum_{n=1}^{\infty}\oplus \HH_n\right )_{\ell_2},\]
where $\HH^n$ is an $n$-dimensional Hilbert space with orthonormal
basis $\{e_{in}\}_{i=1}^n$.  Let $P$ be the orthogonal projection
onto the one dimensional subspace spanned by $\sum_{i=1}^ne_{in}$.
Then $\{(I-P)e_{in}\}_{i=1,n=1}^{n-1,\ \infty}$ as a family of
vectors in $\HH$ is $\delta$-separated, 1-Bessel, but not a
Riesz sequence.  (Careful:  We have thrown away the vectors
$(I-P)e_{nn}$ above.)
\end{remark}
Note also that a $\delta$-seperated sequence may not be Bessel.
\begin{example}
	Let $x_i=e_1+e_{i+1}, i=1, 2, \ldots.$ Then $\{x_i\}_{i=1}^\infty$ is not a Besel sequence. However, it is $\delta$-seperated.
\end{example}
Indeed, let $P_j$ be the projection onto $\overline{\spn}\{x_i\}_{i\not=j}$. Then 
$$\|x_j-P_jx_j\|^2=\|e_1+e_j-P_j(e_1+e_j)\|^2=\|e_j+e_1-P_je_1\|^2=1+\|e_1-P_je_1\|^2\ge1,$$ for all $j$. So $\{x_i\}_{i=1}^\infty$ is $\delta$-separated, where $\delta=1$.

The next proposition presents a fundamental property of separated
sequences.

\begin{proposition}\label{prop7}
	If a family of vectors $\{x_i\}_{i=1}^{\infty}$ is separated, then there are vectors
	$\{y_i\}_{i=1}^{\infty}$ satisfying:
	\[ \langle y_i,x_j\rangle = \delta_{ij},\mbox{ for all i,j}.\]
	If it is $\delta$-separated then, $\underset{i}{\sup}\|y_i\|<\infty$.
\end{proposition}

\begin{proof}
	Fix $j$ and let $P_j$ be the orthogonal projection onto
	$\overline{\spn}\{x_i\}_{i\not= j}$.  Note that $P_jx_j\not= x_j$ and so
	$(I-P_j)x_j \not= 0$.  
	
	Clearly,
	\[ \langle (I-P_j)x_j, x_i\rangle 
 =0 \mbox{ for } i\not=j.\]
	So let
	\[ y_j=\frac{(I-P_j)x_j}{\|(I-P_j)x_j\|^2},\] we get the desired sequence.
	
	For the $\delta$-separated case, we have that $\|(I-P_j)x_j\|\ge 
	\delta$ and the result follows. 
\end{proof}

 For the next result, we will need:

	\begin{proposition} Let $\{x_i\}_{i=1}^\infty$ be a bounded sequence in a Hilbert space $\HH$. The following are equivalent:
		\begin{enumerate}
			\item  $\{x_i\}_{i=1}^\infty$ is $\delta$-separated. 
			\item $\{x_i\}_{i=1}^\infty$ is separated and  $\{x_i\}_{i=n}^\infty$ is $\delta_1$-separated, for some $n\geq 1$.
		\end{enumerate}
	\end{proposition}
	
	\begin{proof}
		We just need to show that $(2)\Rightarrow (1)$.
		So assume $\{x_i\}_{i=1}^{\infty}$ is separated and 
		$\{x_i\}_{i=n}^{\infty}$ is $\delta_1$-separated. Let $P_j$
		be the projection onto $\overline{\spn}\{x_i\}_{i\not= j}$, for $j=1, 2, \ldots,$ and 
		let $Q_j$ be the projection onto $\overline{\spn}\{x_i\}_{n\le i \not= j}$, for $j=n, n+1, \ldots.$  So
		\[ \|(I-Q_j)x_j\|\ge \delta_1,\mbox{ for all }j\ge n.\]
		We need to show that there exists a $\delta >0$ so that
		\[ \|(I-P_j)x_j\|\ge \delta,\mbox{ for all }j\ge 1.\] We will do this in steps.
		
		\vskip10pt
		\noindent {\bf Step 1}: There exists a $\delta_2 >0$ so that
		\[ \|(I-P_j)x_j\|\ge \delta_2,\mbox{ for all }j\ge n.\] 
		\vskip10pt
		
		We will do this by way of contradiction.  So assume there are
		natural numbers $n\leq n_1<n_2<\cdots$ satisfying:
		\[ \|x_{n_j}-P_{n_j}(x_{n_j})\|<\dfrac{1}{j}.\]
		It follows that there are vectors $y_j \in \spn\{x_i\}_{i=1}^{n-1}$
		and $z_j\in \spn \{x_i\}_{n\le i \not= n_j < \infty}$ so that
		$\|x_{n_j}-(y_j+z_j)\|< \frac{1}{j}$.
		\vskip10pt
		\noindent {\bf Claim 1}:  There are an $\epsilon >0$ and $n_0\in \NN$ so that
		$\|y_{j}\|\ge \epsilon$, for all $j\ge n_0$.
		\vskip10pt
		We prove the claim by way of contradiction.  If the claim fails,
		there are integers $j_1<j_2<\cdots$ so that $\|y_{j_k}\|< \frac{1}{k}$ for all $k=1,2,\ldots$.  It follows that
		\[ \|x_{n_{j_k}}- z_{j_k}\|\le \|x_{n_{j_k}}-(z_{j_k}+y_{j_k})\|
		+ \|y_{j_k}\|< \frac{2}{k}, \mbox{ for all }k,\]
		which contradicts the fact that $\{x_i\}_{i=n}^{\infty}$ is
		$\delta$-separated.
		
		\vskip10pt
		\noindent {\bf Claim 2}: There is a constant $K>0$ so that $\|y_j\|\leq K$, for all $j\ge n_0$.
		
		Define 
		$$\gamma=\inf\{ \|u-v\| : u\in \spn\{x_i\}_{i=1}^{n-1}, v\in\overline{\spn}\{x_i\}_{i=n}^\infty, \|u\|=1\}.$$
		We will show that $\gamma>0$. Indeed, if $\gamma=0$ then there are sequences $\{u_j\}_{j=1}^\infty \subset\spn\{x_i\}_{i=1}^{n-1}, \|u_j\|=1$, for all $j$, and $\{v_j\}_{j=1}^\infty\subset \overline{\spn}\{x_i\}_{i=n}^\infty$ so that
		\[\|u_j-v_j\|\to 0 \mbox{ as } j\to \infty.\]
		By switching to a subsequence if necessary, we may assume
		$u_j \rightarrow u \in \spn\{x_i\}_{i=1}^{n-1}$ and $u\not= 0$.
		Since 
		$$\|v_j-u\|\leq \|v_j-u_j\|+\|u_j-u\|,$$ we conclude that $v_j\to u\in \overline{\spn}\{x_i\}_{i=n}^\infty$. Thus, $$u\in \spn\{x_i\}_{i=1}^{n-1}\cap \ \overline{\spn}\{x_i\}_{i=n}^\infty.$$ Since 
		$u\in \spn\{x_i\}_{i=1}^{n-1}, u\not=0$, we can write $u=\sum_{i=1}^{n-1}\alpha_ix_i$ for some scalars $\alpha_i's$ not all zero. Without loss of generality, we can assume $\alpha_1\not=0$. Then
		$$x_1=\dfrac{1}{\alpha_1}\left(u-\sum_{i=2}^{n-1}\alpha_ix_i\right)\in \overline{\spn}\{x_i\}_{i=2}^\infty,$$
		which contradicts the fact that $\{x_i\}_{i=1}^\infty$ is separated. So, $\gamma>0.$
		
		Now we have 
		\[ \| \dfrac{y_j+z_j}{\|y_j\|}\|\ge \gamma, \mbox{ for all }j\geq n_0 ,\] and $\underset{j\ge 1}{\sup}\|x_j\|$ is finte. Therefore, there is some $K>0$ such that
		\[ \|y_j\|\leq \dfrac{1}{\gamma}\|y_j+z_j\|\leq \dfrac{1}{\gamma}(\|y_j+z_j-x_{n_j}\|+ \|x_{n_j}\|)\leq K, \mbox{ for all } j\ge n_0.\]
		The Claim 2 is proven. 
		
		Now since $\epsilon \leq \|y_j\|\leq K$ for all $j\ge n_0$, it has a convergent subsequence $y_{j_k}\to y\in \spn\{x_i\}_{i=1}^{n-1}$, and $y\not=0$.
		
		From the fact that
		\[\|x_{n_{j_k}}-z_{j_k}-y\|\leq \|x_{n_{j_k}}-z_{j_k}-y_{j_k}\|+\|y_{j_k}-y\|\leq \dfrac{1}{j_k}+\|y_{j_k}-y\|,\]
		we conclude $x_{n_{j_k}}-z_{j_k}\to y\in\overline{\spn}\{x_i\}_{i=n}^\infty$ as $k\to \infty$. Thus, 
		\[y\in \spn\{x_i\}_{i=1}^{n-1}\cap \overline{\spn}\{x_i\}_{i=n}^\infty\]
		By the same argument as in the proof of Claim 2, this leads to a contradiction with the fact that $\{x_i\}_{i=1}^\infty$ is separated.
		
		\vskip10pt
		\noindent {\bf Step 2}: There exists a $\delta >0$ so that
		\[ \|(I-P_j)x_j\|\ge \delta,\mbox{ for all }j\ge 1.\] 
		
		Since $\{x_i\}_{i=1}^\infty$ is separated, for each $i=1,2,\ldots,n-1$, there exists $\epsilon_i>0$ so that $\|(I-P_i)x_{i}\|\ge \epsilon_i$.
		Combined with Step 1, we have that $\{x_i\}_{i=1}^{\infty}$
		is $\delta$-separated, where $\delta = \underset{i=1, \ldots, n-1}{\min}\{\epsilon_i, \delta_2\}.$
		The proof of the Proposition is completed. 
\end{proof}

Now we give a complete classification of when the state estimation problem is
solvable for all measurement vectors in $\ell_1$.  Note that we
have done it in complete generality and not assumed that
$\{x_k\}_{k=1}^{\infty}$ is injective.

\begin{theorem}\label{thm105}
	Let $\mathcal{X}=\{x_k\}_{k=1}^{\infty}$ be a frame for the real or complex space $\ell_2$.  The following
	are equivalent:
	\begin{enumerate}
		\item For every real vector $a=(a_1,a_2,\ldots)\in \ell_1$, there is a 
		Hilbert Schmidt self-adjoint operator $T$ so that
		\[ \langle Tx_k, x_k\rangle = a_k,\mbox{ for all }k=1,2,\ldots.\]
		\item The sequence $\{\tilde{x}_k\}_{k=1}^{\infty}$ is $\delta$-separated.
	\end{enumerate}
\end{theorem}

\begin{proof}
$(1)\Rightarrow (2)$: By (1), for each $k=1, 2,\ldots$, there is a Hibert Schmidt self-adjoint operator $R_k$, and hence a vector $\tilde{R}_k\in\tilde{\HH}$ so that 
\[ \langle \tilde{R}_k, \tilde{x}_l\rangle=\langle R_kx_l, x_l\rangle = \begin{cases}
1&\mbox{ if } k=l\\
0&\mbox{ if }k\not= l.
\end{cases}\] 
It follows that $\tilde{x}_l\notin\overline{\spn}\{\tilde{x}_k\}_{k\not=l}$ and hence $\{\tilde{x}_k\}_{k=1}^{\infty}$ is separated. We now proceed by way
of contradiction.  Suppose that  $\{\tilde{x}_k\}_{k=1}^{\infty}$ is not
$\delta$-separated. Then $\{\tilde{x}_k\}_{k=n}^{\infty}$ is not
$\delta_n$-separated for all $n$. Then for $n=1$, there is $k_1\geq 1$ such that 
$$\| \tilde{x}_{k_1} -P_{k_1}(\tilde{x}_{k_1})\|<\dfrac{1}{2}.$$ 
Since $P_{k_1}(\tilde{x}_{k_1})\in \overline{\spn}\{\tilde{x}_{k}\}_{k=1, k\not=k_1}^{\infty}$, there are some scalars  $\alpha_k, k \in I$, where $I$ is a finite subset of $\{k: k\geq 1, k\not=k_1\}$ such that
$$\|P_{k_1}(\tilde{x}_{k_1})-\sum_{k\in I}\alpha_k\tilde{x}_{k}\|<\dfrac{1}{2}.$$ 
Let $y_{1}=\sum_{k\in I}\alpha_k\tilde{x}_{k}$. Then 
$$\|\tilde{x}_{k_1}-y_{1}\|<1.$$ 
Now let $n_2>\max\{k_1, k\}_{k\in I}$. Since $\{\tilde{x}_{k}\}_{k=n_2}^\infty$ is not $\delta_{n_2}$-separated, similar 
to the above, there are numbers $n_2\leq k_2<n_3$ and a vector $$y_{2}\in \spn\{\tilde{x}_{k}: n_2\leq k\not=k_2<n_3\}$$ such that 
$$\|\tilde{x}_{k_2}-y_{2}\|<\dfrac{1}{2^3}.$$
Continuing this procedure we can choose
$1=n_1 \leq k_1 < n_2 \leq k_2 < n_3 < \cdots$ and vectors
\[ y_{m} \in \spn\{\tilde{x}_{k}:n_m\le k\not= k_m < n_{m+1}\},\]
such that $$\|\tilde{x}_{k_m}-y_{m}\|< \frac{1}{m^3},$$ for all $m$.
Now let $a=\{a_k\}_{k=1}^\infty\in \ell_1$, where 
\[a_k =
\begin{cases}
\frac{1}{m^2}&\mbox{ if }k=k_m\\
0 &\mbox{ otherwise }.
\end{cases}\]
Then by assumption, there exists a Hilbert Schmidt self-adjoint operator $T$ and a vector $\tilde{T}\in \tilde{\HH}$ so that 
$\langle \tilde{T}, \tilde{x}_{k}\rangle=\langle Tx_{k}, x_{k}\rangle = a_k$ for all $k$. But then
\[\dfrac{1}{m^2}=\langle \tilde{T}, \tilde{x}_{k_m}\rangle=\langle \tilde{T},\tilde{x}_{k_m}-y_m\rangle\leq \|\tilde{T}\|\|\tilde{x}_{k_m}-y_{m}\|\leq \|\tilde{T}\|\dfrac{1}{m^3},\]
which implies $\|\tilde{T}\|\geq m$ for all $m$, a contradiction.

\vskip10pt
$(2)\Rightarrow (1)$:  Since $\{\tilde{x}_k\}_{k=1}^\infty$ is $\delta$-separated, by Proposition \ref{prop7}, 
there
are vectors $\{\tilde{T}_k\}_{k=1}^{\infty}$ in $\tilde{\HH}$ satisfying
\[ \langle \tilde{T}_k, \tilde{x}_l\rangle=
\begin{cases}
1 &\mbox{ if } k=l\\
0 &\mbox{ if }k\not= l
\end{cases}\]
for all $k, l\geq 1$, and $\underset{k\ge 1}{\sup}\|\tilde{T}_k\|< \infty$.
Now, fix $a=(a_1,a_2,\ldots)\in \ell_1$ and let
\[ \tilde{T} = \sum_{k=1}^{\infty}a_k\tilde{T}_k.\]
This series converges since $a\in \ell_1$ and $\underset{k\ge 1}{\sup}\|\tilde{T}_k\|<\infty$.
Now, let $T$ be the Hilbert Schmidt self-adjoint operator that corresponds with $\tilde{T}$. Then we have 
$$ \langle Tx_k, x_k\rangle =\langle \tilde{T}, \tilde{x}_k\rangle = a_k, \mbox{ for all } k=1,2,\ldots.$$ This completes the proof.
\end{proof}

Now we show that there is no injective frame for which the state
estimation problem is solvable for all measurements taken from $\ell_2$. Note that for a Hilbert Schmidt self-adjoint  operator $T$ on the Hilbert space $\ell_2$, the corresponding vector $\tilde{T}$ is defined as in the proof of Theorem \ref{Real inf} for the real case and Theorem \ref{complex infinite} for the complex case.

\begin{theorem}\label{thm106}
	There is no injective frame $\mathcal{X}=\{x_k\}_{k=1}^{\infty}$ 
	in the real or complex space $\ell_2$ so that for every $a=\{a_k\}_{k=1}^{\infty}\in \ell_2$,
	there is a Hilbert Schmidt operator $T$ so that 
	\[\langle Tx_k, x_k\rangle = a_k,\mbox{ for all }k=1,2,\ldots.\]
\end{theorem}

\begin{proof}
	We will proceed by way of contradiction. The proof is divided into steps.

	Suppose that there is an injective frame $\mathcal{X}=\{x_k\}_{k=1}^{\infty}$ for which the state estimation problem is solvable for all choices $\{a_k\}_{k=1}^\infty\in \ell_2$.
	\vskip10pt
	\noindent {\bf Step I}: There are vectors $\tilde{R}_k\in \tilde{\HH}, k=1, 2, \ldots$ so that $\langle \tilde{R}_k,\tilde{x}_l\rangle=\delta_{kl}$. 
	\vskip10pt
	This is immediate because by assumption, for each $k=1, 2,\ldots$, there is a Hilbert Schmidt self-adjoint operator $R_k$ so that 
	\[ \langle \tilde{R}_k,\tilde{x}_l\rangle =\langle R_kx_l,x_l\rangle =\begin{cases}
	1&\mbox{ if } k=l\\
	0&\mbox{ if } k\not= l.
	\end{cases}\] 
	
	\vskip10pt
	Denote $E_n=\spn\{\tilde{x}_k\}_{k=1}^n$ and
	let $P_n$ be the projection onto $E_n$.
	
	\vskip10pt
	\noindent {\bf Step II}: If there is a real vector $\{a_k\}_{k=1}^\infty\in \ell_2$ satisfying $\underset{n}{\sup}\|\sum_{k=1}^na_k\tilde{R_k}\|=\infty$, then there is a real vector $\{b_k\}_{k=1}^{\infty}\in\ell_2$ and $n_1<n_2< \cdots$ so that
	\[ \|P_{n_j}(\sum_{k=1}^{n_j}b_k\tilde{R_k})\|\ge j.\]
	\vskip10pt
	Indeed, since $\underset{n}{\sup}\|\sum_{k=1}^na_k\tilde{R_k}\|=\infty$, we can choose a sequence $m_1<m_2 <\cdots$ so that
	\[ \|\sum_{k=1}^{m_j}a_k\tilde{R}_k\|\ge 2j.\]
	
	For any $j>1$, we have 
	\begin{align*}
	\|\sum_{k=1}^{m_1}a_k\tilde{R}_k-\sum_{k=m_1+1}^{m_j}a_k\tilde{R}_k\|&\geq
	\|\sum_{k=1}^{m_j}a_k\tilde{R}_k\|-2\|\sum_{k=1}^{m_1}a_k\tilde{R}_k\|\\
	&\geq 2j-2\|\sum_{k=1}^{m_1}a_k\tilde{R}_k\|.
	\end{align*} 
	
	Combining this with the fact that $E_1\subset E_2 \subset \ldots$ and $\cup_{n=1}^{\infty}
	E_n$ is dense in $\tilde{\HH}$, 
	we can choose $j$ large enough so that
	\[ \|P_{m_{j}}(\sum_{k=1}^{m_1}a_k\tilde{R}_k)\|\ge \frac{1}{2}\|\sum_{k=1}^{m_1}a_k\tilde{R_k}\|,\]
	and 
	$$\|\sum_{k=1}^{m_1}a_k\tilde{R}_k-\sum_{k=m_1+1}^{m_j}a_k\tilde{R}_k\|\geq 4.$$
	Since 
	\begin{align*}
	\|P_{m_j}(\sum_{k=1}^{m_1}a_k\tilde{R}_k)&+P_{m_j}(\sum_{k=m_1+1}^{m_j}a_k\tilde{R}_k)\|^2+
	\|P_{m_j}(\sum_{k=1}^{m_1}a_k\tilde{R}_k)-P_{m_j}(\sum_{k=m_1+1}^{m_j}a_i\tilde{R}_k)\|^2\\
	&= 2\left (\|P_{m_j}(\sum_{k=1}^{m_1}a_k\tilde{R}_k)\|^2+\|P_{m_j}(\sum_{k=m_1+1}^{m_j}a_k\tilde{R_k})\|^2\right )\\
	&\ge 2\|P_{m_j}(\sum_{k=1}^{m_1}a_k\tilde{R}_k)\|^2,
	\end{align*} 
	we can choose $b_i=a_i$ for $i=1,\ldots, m_1$ and $b_i\in \{a_i, -a_i\}$ for $i=m_1+1,\ldots, m_j$ so that 
	$$\|P_{m_j}(\sum_{k=1}^{m_j}b_k\tilde{R}_k)\|\ge \|P_{m_j}(\sum_{k=1}^{m_1}b_k\tilde{R}_k)\|\ge\dfrac{1}{2}\|\sum_{k=1}^{m_1}b_k\tilde{R_k}\|\ge 1 \mbox{ and } \|\sum_{k=1}^{m_j}b_k\tilde{R}_k\|\geq 4.$$
	Setting $n_1=m_j$, 
	$$\|P_{n_1}(\sum_{k=1}^{n_1}b_k\tilde{R}_k)\|\ge 1 \mbox{ and } \|\sum_{k=1}^{n_1}b_k\tilde{R}_k\|\geq 4.$$
	
	Now for $m_j$ above, by the same argument, there is $m_l>m_j$ and $b_i\in \{a_i, -a_i\}$ for $ i=m_{j}+1,\ldots ,m_l$ so that 
	$$\|P_{m_l}(\sum_{k=1}^{m_l}b_k\tilde{R}_k)\|\geq \|P_{m_l}(\sum_{k=1}^{m_j}b_k\tilde{R}_k)\|\geq\dfrac{1}{2} \|\sum_{k=1}^{m_j}b_k\tilde{R}_k\|\geq 2$$ and 
	$$\|\sum_{k=1}^{m_l}b_k\tilde{R}_k\|\geq 6.$$
	Set $n_2=m_l$ we get 
	$$\|P_{n_2}(\sum_{k=1}^{n_2}b_k\tilde{R}_k)\|\ge 2.$$
	Continuing this process inductively, the result follows.
	
	\vskip10pt
	\noindent {\bf Step III}: For all vectors $\{ a_k\}_{k=1}^{\infty}\in \ell_2$, $\underset{n}{\sup}\|\sum_{k=1}^na_k\tilde{R_k}\|$ is finite.
	\vskip10pt
	Suppose by contradiction that there is a vector $\{ a_k\}_{k=1}^{\infty}\in \ell_2$ so that $\underset{n}{\sup}\|\sum_{k=1}^na_k\tilde{R_k}\|=\infty$. Let $\{b_k\}_{k=1}^\infty$ be the vector in Step II, then there exists a vector
	$\tilde{T}\in \tilde{\HH}$ so that $\langle \tilde{T},\tilde{x}_k\rangle = b_k$, for all
	$k=1,2,\ldots$.  It follows that $$P_{n_j}\tilde{T}= 
	P_{n_j}(\sum_{k=1}^{n_j}b_k\tilde{R}_k)$$ for all
	$j=1,2,\ldots$.  Hence, 
	\[ \infty = \sup_{j}\|P_{n_j}(\sum_{k=1}^{n_j}b_k\tilde{R}_{k})\|=\sup_{j}\|P_{n_j}\tilde{T}\|\le \|\tilde{T}\|,\]
	which is a contradiction.
	
	\vskip10pt
	\noindent {\bf Step IV}: $\{ \tilde{R}_k\}_{k=1}^{\infty}$ is a Bessel sequence in $\tilde{\HH}$. 
	\vskip10pt

	For each $n\in \NN$, define an operator
	\begin{align*} T_n: \ell_2 &\longrightarrow \tilde{\HH}\\
	x=(a_1, a_2, \ldots)&\longmapsto  T_n(x)=\sum_{k=1}^{n}a_k\tilde{R}_k
	\end{align*}
	Then $T_n$ is a bounded linear operator for all $n$. 
	
	By Step III, $\underset{n}{\sup}\|\sum_{k=1}^na_k\tilde{R_k}\|$ is finite for all $x=\{a_k\}_{k=1}^\infty$.  
	By the Uniform Boundedness Principle, $\underset{n}{\sup}\|T_n\|\le B$, for some $B>0$. For any $n, m \in \NN, m>n$, we have
	$$\|\sum_{k=n+1}^m{a_k\tilde{R}_k}\|^2=\|T_m(\sum_{k=n+1}^{m}a_ke_k)\|^2\leq B^2\sum_{k=n+1}^{m}a_k^2.$$ 
	It follows that $\sum_{k=1}^{\infty}a_k\tilde{R}_k$ converges, and hence $\{\tilde{R}_k\}_{k=1}^\infty$ is Bessel.
	
	\vskip10pt
	\noindent {\bf Step V}:  We arrive at a contradiction.
	\vskip10pt
	
	We have shown that under our assumption, $\{\tilde{R}_k\}_{k=1}^{\infty}$ is $B^2$-Bessel
	for some $B$.  Now
	choose any $a=\{a_k\}_{k=1}^{\infty}\in \ell_2$.  We have that
	\[ \|\sum_{k=1}^{\infty}a_k\tilde{R}_k\|^2\le B^2 \sum_{k=1}^{\infty}
	a_k^2.\]
	By Theorem \ref{thm103}, $\sum_{k=1}^{\infty}a_k\tilde{x}_k$ converges. Now, we have 
	\begin{align*}
	\|\sum_{k=1}^{\infty}a_k\tilde{x}_k\| &= \sup_{\|x\|\le 1}
	|\langle x,\sum_{k=1}^\infty a_k\tilde{x}_k\rangle|\\
	&\ge \frac{1}{B\|a\|}|\langle \sum_{k=1}^\infty a_k\tilde{R}_k,\sum_{l=1}^\infty
	a_l\tilde{x}_l\rangle|\\
	&=\frac{1}{B\|a\|} |\sum_{k,l=1}^\infty a_ka_l\langle \tilde{R}_k,\tilde{x}_l\rangle|\\
	&= \frac{1}{B}\|a\|.
	\end{align*}
	It follows that $\{\tilde{x}_k\}_{k=1}^{\infty}$ has
	a positive lower Riesz bound and since this family is
	injective, it is a Riesz basis.
	Hence by Theorem \ref{thm103}, it is a frame sequence. But then by Corollary \ref{Coro10}, $\{x_k\}_{k=1}^\infty$ cannot be injective, a contradiction. The proof of our theorem is now complete.
\end{proof}

\begin{remark}
As in the finite dimensional case, it is often the case that the
state estimation problem is not solvable.  But again there is a
natural way to get a good estimation to the solution.  Given
a frame $\{x_k\}_{k=1}^{\infty}$ and $\{a_k\}_{k=1}^{\infty}\in\ell_2$, choose $m$ so that $\sum_{k=m+1}^{\infty}a_k^2\le \epsilon$.
Then apply the argument in the finite case to get the best solution
for $\{a_k\}_{k=1}^m$.
\end{remark}


\begin{thebibliography}{WW}
\bibitem{BK}  J.J. Benedetto and A. Kebo, {\it The role of frame
force in quantum detection}, Journal of Fourier Analysis and
Applications {\bf 14} (2008) 443-474.

\bibitem{BH}  B. Bodmann and J. Haas, {\it A short history of frames
and quantum designs}, Preprint.  arXiv:1709.01958.

\bibitem{CaK}  P.G. Casazza and N.J. Kalton, {\it Generalizing the
Paley-Wiener perturbation theory for Banach spaces}, Proceedings of
the AMS {\bf vol. 127 No. 2} (1999) p. 519-527.

\bibitem{CK}  P.G. Casazza and G. Kutyniok, Editors {\it Finite Frames: Theory and Applications}, Birkhauser, Boston (2012).

\bibitem{C}  P.G. Casazza and M. Leon, {\it Existence and Construction of finite frames with a given frame operator}, 
International Journal of Pure and Applied Mathematics, Vol.
{\bf 63} No. 2 (2010) 149-158.

\bibitem{C1}  P.G. Casazza and R. Lynch {\it A brief introduction to Hilbert space frame theory and
its applications}, Proceedings of Symposia in Applied Mathematics - Finite Frame Theory; AMS Short Course 2015, K. Okoudjou, Ed. {\bf 73} (2016) 1-51.arxiv:  1509.07347

\bibitem{CPT}  P.G. Casazza, E. Pinkham, and B. Toumanen,
{\it Riesz outer product Hilbert space frames:  quantitative bounds, topological properties and full geometric characterization}, Jour. Math Anal and Appls, {\bf 441} No. 1 (2016) 475-498. arxiv:  1410.7755.
http://dx.doi.org/10.1016/j.jmaa.2016.04.001

\bibitem{Ch}  O. Christensen, {\it Frames and Riesz bases}, Birkhauser, Boston, (2016).

\bibitem{Eld3} Y.C. Eldar and H. Bolcskei, {\it Geometrically uniform  
frames}, IEEE Transactions on Information Theory {\bf 49} (4) (2003)
993-1006.

\bibitem{Eld}  Y.C. Eldar, and G.D. Forney, Je. {\it Optimal tight frames and quantum measurement}, IEEE Tranactions on Information
Theory, {\bf 48} No. 3, (2002) 599-610.

\bibitem{Eld2}  Y.C. Eldar, {\it Von Neumann measurement is optimal
for detecting linearly independent 0 quantum states}, Phys. Rev.
A (3) {\bf 68} (5) (2003).



\bibitem{HL}  D. han, D.R. Larson, B. Liu, and R. Lin, {\it operator-valued measures, dilatons, and the theory of frames}, Memoirs of
AMS, {\bf 229} No. 1075, (2013).

\bibitem{HW}  C.W. Hauladen and W.K. Wooters, {\it A "pretty good"
measurement for distinguishing quantum states}, J. Modern Opt.
{\bf 41} (12) (1994) 2385-2390.

\bibitem{H}  C.W. Helstrom, {\it Quantum detection and estimation
theory}, J. Statist. Phys. {\bf 1} (1969) 231-252.

\bibitem {Coc}  B. Moran, S. Howard, and D. Cochran, {\it Positive-operator-valued measures:  A general setting for frames}, Excursions in Harmonic Analysis book Series, {\bf 2} (2012) 49-64.

\bibitem{P}  A. Peres and D.R. Terno, {\it Optimal distinction
between non-orthogonal quantum states}, J. Phys. A {\bf 31}
(34) (1998) 7105-7111.

\bibitem{S}  A.J. Scott, {\it Tight informationally complete
quantum measurements}, J. Math Physics {\bf 39} No. 42 (2006)
13507-13530.

\bibitem{Y}  H.P. Yuen, R.S. Kennedy, and M. Lax, {\it Optimum testing of multiple hypotheses in quantum detection theory},
IEEE Transactions on Information Theory {\bf IT-21} (1975) 125-134.`
\end{thebibliography}
\end{document}